\documentclass[11pt,a4paper]{article}

\usepackage{amsmath,amssymb,amscd}
\usepackage{amsthm,color,bm}
\usepackage{mathtools}
\usepackage{graphicx}
\usepackage{float}
\usepackage{tikz}
\usetikzlibrary{cd}
\usepackage{comment}
\usepackage{mathrsfs}
\usepackage{enumerate}

\usepackage{multienum}
\usepackage{longtable}
\usepackage{array}
\usepackage{cases}
\usepackage{here}
\usepackage{amsfonts}
\usepackage{stmaryrd}
\usepackage{fancybox}
\usepackage{multirow,bigdelim}
\usepackage{arydshln}
\usepackage{ascmac}
\numberwithin{equation}{section}
\allowdisplaybreaks[1]

\usepackage{mathabx}

\usepackage{amsfonts}

\usepackage[top=3cm, bottom=3cm, left=2.5cm, right=2.5cm]{geometry}

\theoremstyle{plain}
\newtheorem{theorem}{Theorem}[section]
\newtheorem{proposition}[theorem]{Proposition}

\newtheorem{corollary}[theorem]{Corollary}

\theoremstyle{definition}
\newtheorem{definition}[theorem]{Definition}
\newtheorem{example}[theorem]{Example}
\newtheorem{remark}[theorem]{Remark}

\newcommand{\RR}{\mathbb{R}}

\newcommand{\dd}{\mathsf{d}}
\newcommand{\Ad}{\mathrm{Ad}}
\newcommand{\ad}{\mathrm{ad}}

\def\Set#1{\Setdef#1\Setdef}
\def\Setdef#1|#2\Setdef{\left\{#1\,\;\mathstrut\vrule\,\;#2\right\}}%

\begin{document}

\title{Stability analysis for the pseudo-Riemannian geodesic flows of step-two nilpotent Lie groups}
\author{Genki Ishikawa\footnotemark[1] \and Daisuke Tarama\footnotemark[1]}
\begingroup
\def\thefootnote{\fnsymbol{footnote}}
\footnotetext[1]{Department of Mathematical Sciences, Ritsumeikan University. 1-1-1 Nojihigashi, Kusatsu, Shiga, 525-8577, Japan.}
\endgroup
\date{\empty}
\maketitle


\begin{abstract}
The present paper deals with the stability analysis for the geodesic flow of a step-two nilpotent Lie group equipped with a left-invariant pseudo-Riemannian metric.
The Lie-Poisson equation is formulated using the so-called $j$-mapping, a linear operator associated to the step-two nilpotent Lie algebras equipped with the induced scalar product. 
The stability of equilibrium points for the Hamilton equation is determined in terms of their Williamson types. 
\end{abstract}

\section{Introduction}

It is widely known that the geodesics of a Riemannian manifold play very important roles in investigating the geometry. 
In particular, Lie groups equipped with a left-invariant Riemannian metric form a reasonable class of Riemannian manifolds to be studied more concretely from geometrical points of view. 
It is also well known that the geodesics induce a Hamiltonian system on the (co)tangent bundle of the Riemannian manifold, known as the {\it geodesic flow}. 
The same framework for formulating the geodesic flows applies also to pseudo-Riemannian manifolds. 
It is, then, of much interest to study the geodesic flows of Lie groups with respect to left-invariant both Riemannian and pseudo-Riemannian metrics. 

On the other hand, the geodesic flow of the three-dimensional rotation group $SO(3)$ with respect to a left-invariant Riemannian metric can physically be interpreted as a free rigid body, i.e. the rotational motion of a rigid body under no external force. 
Motivated by the studies on infinite-dimensional integrable systems such as the Korteweg-de Vries equation, the free rigid body dynamics were extended to more general Lie groups from the viewpoint of Hamiltonian systems and the generalizations are in fact completely integrable geodesic flows of the Lie groups with respect to certain left-invariant metrics. 
In particular, Mishchenko and Fomenko introduced a class of left-invariant metrics on an arbitrary semi-simple Lie group which allow the completely integrable geodesic flow around 1980. 
See \cite{dikii_1972,manakov_1976,Mishchenko,Mishchenko-Fomenko,ratiu_1980}, as well as \cite{fomenko_1988,perelomov_1990} for more details. 

From the viewpoint of dynamical systems theory, it is an important issue to discuss the stability of equilibria for rigid body systems. 
See e.g. \cite[Appendices 2 \& 5]{arnold_1989} for more information about the topics. 
However, the stability of equilibria for the generalized free rigid body dynamics has been studied only since the middle of 2000's. 
Up to now, there have been many researches in this direction, but they mostly focus on the semi-simple Lie groups. 
See \cite{feher-marshall_2003,spiegler_2006,bolsinov_oshemkov_2009, birtea-casu-ratiu-turhan_2012,izosimov_2014,izosimov_2017, bolsinov_izosimov_2014, ratiu_tarama_2015, ratiu_tarama_2020}. 

An advantage to introduce the Hamiltonian formalism for the geodesics can be found in the description of the geodesic flow in terms of the {\it Lie-Poisson equation} on (the dual to) the Lie algebra, which is given due to the symmetry.
It generalizes the Euler equation for the free rigid body and can be formulated in view of Lie-Poisson Reduction Theorem. 
See \cite{marsden_ratiu_1999,ratiu_et_al_2005} for the details about Lie-Poisson Reduction. 

\medskip

The present paper deals with the stability of relative equilibria for the geodesic flow of step-two nilpotent Lie groups equipped with a left-invariant (pseudo-)Riemannian metric. 
Extensions of the researches on left-invariant geodesic flows in the direction to other types of Lie groups than semi-simple Lie groups have already been found in \cite{Mishchenko-Fomenko}, which mentions solvable groups. 
For the three-dimensional Lie groups, \cite{holm_marsden_1991} discusses such an extension at the level of classical free rigid body dynamics and it is also used in the analysis of quantum systems in \cite{iwai_tarama_2010}. 

Contrary to the case of semi-simple Lie groups, there are only less unified approaches to the nilpotent Lie groups and hence it seems rather difficult to study the geodesic flows on general nilpotent Lie groups. 
However, it is still possible to handle the dynamical properties of the geodesic flows of the step-two nilpotent Lie groups by a systematic method. 

The Lie-Poisson equation for the geodesic flow of a step-two nilpotent Lie group is described in terms of the so-called $j$-mapping. 
It plays the central role in proving the complete integrability for the geodesic flows for Heisenberg Lie groups and pseudo-$H$-type Lie groups, as studied in \cite{kocsard-ovando-reggiani_2016,bauer_tarama_2018}. 

As for the equilibrium on a general coadjoint orbit in the dual to the Lie algebra for the step-two nilpotent Lie group, on which the Lie-Poisson equation can naturally be restricted as a Hamiltonian system with respect to the orbit symplectic form, the stability analysis amounts to the eigenvalue problem for the $j$-mapping. 
For some typical examples of step-two nilpotent Lie groups, the eigenvalues can be classified and the stability of the equilibrium points is found in terms of the Williamson types. 

In the case where the left-invariant metric is Riemannian, the eigenvalues of the $j$-mapping are purely imaginary and all the equilibrium points are Lyapunov stable. 
However, when the metric is pseudo-Riemannian, there appear different types of eigenvalues for the $j$-mapping and the stability of the equilibria is more complicated. 
Nevertheless, as a main result of the present paper, it is shown that the problem can be solved on a technical condition by virtue of the classification of the Cartan subalgebras in simple Lie algebras of types $\mathsf{B}$ and $\mathsf{D}$, which was carried out by Kostant \cite{kostant_1955,kostant_2009} and Sugiura \cite{sugiura_1959} during 1950's. 

\medskip

The present paper is organized as follows:

\noindent
Section \ref{Left-invariant geodesic flows of a Lie group} gives a review on the Lie-Poisson formalism for the geodesic flow of a Lie group equipped with a left-invariant (pseudo-)Riemannian metric. 
The Lie-Poisson equation is more intensively explained for step-two nilpotent Lie groups in terms of the $j$-mapping. 
Moreover, several important classes of step-two nilpotent Lie groups are mentioned. 

In Section \ref{Equilibrium points and stability}, the stability of equilibrium points for the left-invariant pseudo-Riemannian geodesic flow of a step-two nilpotent Lie group is analysed. 
As preliminaries, the notion of Williamson type of an equilibrium point for a Hamiltonian system is reviewed and related to the Lyapunov stability. 
In the case of general step-two nilpotent Lie groups, the Williamson types of the equilibrium points are determined on the basis of the classification results on the conjugacy classes of Cartan subalgebras of real semi-simple Lie algebras of types $\mathsf{B}$ and $\mathsf{D}$. 
It is based on the method by \cite{sugiura_1959} as it gives more concrete matrix realizations of the conjugacy classes of Cartan subalgebras. 
According to the description of step-two nilpotent Lie groups given in Section \ref{Left-invariant geodesic flows of a Lie group}, the Williamson types of equilibria are more explicitly determined. 

Appendix summarizes the classification for the conjugacy classes of Cartan subalgebras of real semi-simple Lie algebras of types $\mathsf{B}$ and $\mathsf{D}$ along the line of \cite{sugiura_1959}.

\section{Left-invariant geodesic flows of a Lie group}\label{Left-invariant geodesic flows of a Lie group}
In this section, we review the Hamiltonian formalism for the geodesic flow of a Lie group with respect to a left-invariant (pseudo-)Riemannian metric. 
By Lie-Poisson Reduction, the system can be reduced onto the dual to the Lie algebra and the reduced system can be described by the Lie-Poisson equation. 
Then, we focus on the step-two nilpotent Lie groups equipped with a left-invariant (pseudo-)Riemannian metric and describe the Lie-Poisson equation explicitly on the Lie algebra which is identified with the dual space. 
See \cite{abraham_marsden,bauer_tarama_2018,ratiu_tarama_2020} for the details on the Lie-Poisson formulation of the geodesic flows of Lie groups. 

\subsection{Review on the Lie-Poisson equation}
Let $G$ be a Lie group. 
By $\mathfrak{g}$, we denote the corresponding Lie algebra. 
Through the left-translations, the tangent and the cotangent bundles of $G$ are trivialized as 
\begin{align*}
  &TG\supset T_q G\ni X \mapsto \left( q,\left( \dd L_{q^{-1}} \right)_{q}X \right)\in G\times\mathfrak{g},
  \\
  &T^{\ast}G\supset T_q^{\ast}G\ni \xi \mapsto \left( q, \left( L_q \right)^{\ast} \xi \right)\in G\times\mathfrak{g}^{\ast},
\end{align*}
where $q\in G$ and $L_q\colon G\ni h\mapsto q h\in G$. 

We assume that the Lie group $G$ is equipped with a left-invariant (pseudo-)Riemannian metric $\displaystyle \left\langle \cdot, \cdot\right\rangle=\left(\left\langle \cdot, \cdot\right\rangle_q\right)_{q\in G}$, where $\langle \cdot, \cdot\rangle_q: T_q G\times T_q G\rightarrow \mathbb{R}$ is a scalar product, i.e. a non-degenerate symmetric bilinear form. 
The (pseudo-)Riemannian metric $\langle \cdot, \cdot\rangle$ can be identified with the scalar product $\langle \cdot, \cdot\rangle_{e}$ on $\mathfrak{g}\cong T_eG$, where $e\in G$ denotes the unit, through the left-trivialization $T G\cong G\times \mathfrak{g}$.
We write the scalar product $\langle \cdot, \cdot\rangle_e: \mathfrak{g}\times \mathfrak{g}\rightarrow \mathbb{R}$ as $\langle \cdot, \cdot\rangle$ for brevity. 

Through the (pseudo-)Riemannian metric $\langle \cdot, \cdot\rangle$, the tangent bundle $TG$ and the cotangent bundle $T^{\ast}G$ are identified and this identification is compatible with the above left-trivializations of the bundles, $TG\cong G\times \mathfrak{g}$ and $T^{\ast}G\cong G\times \mathfrak{g}^{\ast}$. 
Here, $G\times \mathfrak{g}(\cong TG)$ and $G\times \mathfrak{g}^{\ast}(\cong T^{\ast}G)$ are identified through the scalar product $\langle \cdot, \cdot\rangle: \mathfrak{g}\times \mathfrak{g}\rightarrow \mathbb{R}$ as 
\begin{align}\label{identification}
  G\times\mathfrak{g}\ni \left( q,Y \right)\mapsto \left( q, \left\langle Y,\cdot \right\rangle \right)\in G\times\mathfrak{g}^{\ast}.
\end{align}

On the cotangent bundle $T^{\ast}G$, the canonical one-form $\Theta$ is defined through $\Theta_{\alpha_q}\left(\widetilde{X}\right)=\alpha_q\left(\mathsf{d}\pi_{\alpha_q}\widetilde{X}\right)$, where $\alpha_q\in T_q^{\ast}G\subset T^{\ast}G$, $\widetilde{X}\in T_{\alpha_q}\left(T^{\ast}G\right)$, and $\pi:T^{\ast}G\rightarrow G$ is the canonical projection. 
Then, the canonical symplectic form $\Omega$ is defined by $\Omega=-\mathsf{d}\Theta$. 
Through the identification $T^{\ast}G\cong G\times \mathfrak{g}^{\ast}$, we describe the canonical one-form $\Theta$ and the canonical symplectic form $\Omega$ on $G\times \mathfrak{g}^{\ast}$ as 
\begin{align*}
  \Theta_{\left(q, \mu \right)}\left( U,\nu \right)&=\mu\left( U \right),
  \\
  \Omega_{\left(q, \mu \right)}\left( \left( U,\nu \right),\left( U',\nu' \right) \right)&=\nu'\left( U \right) -\nu\left( U' \right) +\mu\left( \left[ U,U' \right] \right),
\end{align*}
where $\left(q, \mu \right)\in G\times\mathfrak{g}^{\ast}\cong T^{\ast}G$ and $\left( U, \nu \right),\left( U', \nu' \right)\in \mathfrak{g}\times\mathfrak{g}^{\ast}\cong T_{q}G\times\mathfrak{g}^{\ast}\cong T_{\left( q, \mu \right)}\left( T^{\ast}G \right)$. 
Here, the bracket $\left[ \cdot,\cdot \right]$ stands for the Lie bracket on $\mathfrak{g}$. 
See \cite[Proposition 4.4.1]{abraham_marsden} for the proof of the formula. 

Further, using the identification \eqref{identification}, we give another expression of the canonical one-form $\Theta$ and the canonical symplectic form $\Omega$ on $G\times \mathfrak{g}(\cong TG)$. 
Namely, we have 
\begin{align*}
  \Theta_{\left( q, Y \right)}\left( U,V \right)&=\left\langle Y,U \right\rangle, 
  \\
  \Omega_{\left( q, Y \right)}\left( \left( U,V \right),\left( U',V' \right) \right)&=\left\langle V',U\right\rangle - \left\langle V,U' \right\rangle + \left\langle Y, \left[ U,U' \right] \right\rangle,
\end{align*}
where $\left( q, Y \right)\in G\times \mathfrak{g}\cong TG$ and $\left( U,V \right),\left( U',V' \right)\in \mathfrak{g}\times\mathfrak{g}\cong T_qG\times \mathfrak{g}\cong T_{(q,Y)}\left(TG\right)$. 

The (pseudo-)Riemannian metric $\langle\cdot, \cdot \rangle$ on $G$ induces the one on the tangent bundle $TG\cong G\times\mathfrak{g}$ described as
\begin{align}
  \left\langle \left\langle \left( U,V \right),\left( U',V' \right) \right\rangle \right\rangle = \left\langle U,U' \right\rangle + \left\langle V,V' \right\rangle,
  \label{induced metric on TG}
\end{align}
where $\left( U,V \right),\left( U',V' \right)\in \mathfrak{g}\times\mathfrak{g}\cong T_{q}G\times\mathfrak{g}\cong T_{\left( q, Y \right)}\left( TG \right)$.

The Hamiltonian vector field $\widetilde{\Xi}_{H}$ for the Hamiltonian function $H\in\mathcal{C}^\infty\left( G\times \mathfrak{g} \right)$ is given as
\begin{align*}
  \widetilde{\Xi}_{H}\left( q,Y \right)=\left( V,\left( \ad_{V} \right)^{\mathrm{T}}Y-U \right).
\end{align*}
Here, $\mathrm{grad}_{\left( q, Y \right)}H=\left( U,V \right)\in\mathfrak{g}\times \mathfrak{g}\cong T_qG\times \mathfrak{g}$ denotes the gradient vector of $H$ at $\left( q, Y \right)\in TG\cong G\times\mathfrak{g}$ with respect to the (pseudo-)Riemannian metric \eqref{induced metric on TG} on $TG\cong G\times\mathfrak{g}$:
\begin{align*}
  \left\langle \left\langle \mathrm{grad}_{\left( q, Y \right)}H,\left( U',V' \right) \right\rangle \right\rangle=\left( \dd H \right)_{\left( q, Y \right)}\left( U', V' \right), 
  \quad \left( U',V' \right)\in\mathfrak{g}\times\mathfrak{g}\cong T_qG\times \mathfrak{g}.
\end{align*}
On the other hand, $\left( \ad_{V} \right)^\mathrm{T}\in\mathrm{End}(\mathfrak{g})$ stands for the adjoint operator of $\ad_{V}$ with respect to the scalar product $\langle\cdot,\cdot\rangle$ on $\mathfrak{g}$:
\begin{align}\label{eq_adT}
  \left\langle \left( \ad_{V} \right)^{\mathrm{T}}Y,Z \right\rangle = \left\langle Y,\ad_{V}Z \right\rangle, 
  \quad Y,Z\in\mathfrak{g}. 
\end{align}

The Poisson bracket $\left\{\cdot, \cdot\right\}$ associated to the canonical symplectic form $\Omega$ is defined by $\{F, F^{\prime}\}=\Omega\left(\widetilde{\Xi}_{F}, \widetilde{\Xi}_{F^{\prime}}\right)$, $F, F^{\prime}\in \mathcal{C}^{\infty}(T^{\ast}G)$. 
On $G\times \mathfrak{g}(\cong TG\cong T^{\ast}G)$, it is written as 
\begin{align*}
  \left\{ F,F' \right\}\left( q,Y \right)=-\left\langle U',V \right\rangle +\left\langle U,V' \right\rangle - \left\langle Y, \left[ V,V' \right] \right\rangle,
\end{align*}
where $F,F'\in\mathcal{C}^\infty\left( G\times \mathfrak{g} \right)$ and $\left( U,V \right)=\mathrm{grad}_{\left( q,Y \right)}F,\,\left( U',V' \right)=\mathrm{grad}_{\left( q,Y \right)}F'\in\mathfrak{g}\times\mathfrak{g}\cong T_qG\times \mathfrak{g}$. 
If the Hamiltonian $H\in\mathcal{C}^\infty\left( TG \right)$ is left-invariant, i.e. if it does not depend on $q\in G$ as a function on $G\times \mathfrak{g}\left(\cong TG\right)$, we have $\mathrm{grad}_{\left( q,Y \right)}H=\left( 0,V \right)$, and thus the Hamilton equation\footnote{In\cite[p.497]{bauer_tarama_2018}, the Hamilton equation should be read as in \eqref{hamilton_eq} of the present paper with $q$ being replaced by $p$.} is written as
\begin{align}\label{hamilton_eq}
  \begin{cases}
    \displaystyle\frac{\dd q}{\dd t} = \left( \dd L_{q} \right)_{e} V,
    \\[7pt]
    \displaystyle\frac{\dd Y}{\dd t} = \left( \ad_{V} \right)^{\mathrm{T}}Y.
  \end{cases}
\end{align}
The Hamiltonian function $H$ depends only on $Y$ and hence so is $V$. 
Thus, the first component of \eqref{hamilton_eq} can be solved in terms of $q\in G$ by using the solution in $Y$ to the second component. 

The second component of \eqref{hamilton_eq} is regarded as a Hamilton equation also with respect to the Lie-Poisson bracket on $\mathfrak{g}\cong \mathfrak{g}^{\ast}$. 
The Lie-Poisson bracket on $\mathfrak{g}$ is defined through $\{f, f^{\prime}\}(Y)=-\left\langle Y, \left[\nabla f(Y), \nabla f^{\prime}(Y)\right]\right\rangle$, $f, f^{\prime}\in \mathcal{C}^{\infty}(\mathfrak{g})$. 
Here, $\nabla f(Y)\in \mathfrak{g}$ is the gradient vector of $f$ at $Y\in \mathfrak{g}$ with respect to the scalar product $\langle \cdot, \cdot \rangle$: $\langle \nabla f(Y), Z\rangle =\left(\mathsf{d}f\right)_Y\left(Z\right)$ for all $Z\in \mathfrak{g}$. 
With respect to the Lie-Poisson bracket, we define the Hamiltonian vector field $\Xi_f$ for the Hamiltonian $f\in \mathcal{C}^{\infty}(\mathfrak{g})$ through 
$\Xi_f(f^{\prime})=-\{f, f^{\prime}\}$ for all $f^{\prime}\in \mathcal{C}^{\infty}(\mathfrak{g})$. 
With respect to the Lie-Poisson bracket, the second component of \eqref{hamilton_eq}, known as the Lie-Poisson equation, is the Hamilton equation for the Hamiltonian $h$ on $\mathfrak{g}$, which is induced from the left-invariant function $H$ on $TG\cong G\times \mathfrak{g}$. 
This reduction procedure is called the \textit{Lie-Poisson Reduction}. 
See \cite{ratiu_et_al_2005} for the details of Lie-Poisson Reduction. 

It is a well-known fact that a Poisson manifold can be stratified into the disjoint union of symplectic manifolds. 
In the case of a Lie-Poisson space, i.e. the dual space to a Lie algebra equipped with the Lie-Poisson bracket, is stratified by coadjoint orbits on which the orbit symplectic forms are naturally induced. 
(See e.g. \cite{marsden_ratiu_1999} for the details about the symplectic stratifications.) 
In the current situation, through the identification $\mathfrak{g}\ni Y \mapsto \langle Y,\cdot\rangle\in\mathfrak{g}^{\ast}$, the stratification of the Lie-Poisson space $\left(\mathfrak{g}, \left\{\cdot,\cdot\right\}\right)$ is given by the coadjoint orbits 
\begin{align*}
  \mathcal{O} = \Set{\left( \Ad_{q}  \right)^{\mathrm{T}}Y | q\in G},
\end{align*}
where $Y\in \mathfrak{g}$. 
On the leaf $\mathcal{O}$, the orbit symplectic form is induced through 
\begin{align*}
  \omega_{Y}\left( \left( \ad_U \right)^{\mathrm{T}}Y,\left( \ad_{U'} \right)^{\mathrm{T}}Y\right)=-\left\langle Y,\left[ U,U' \right]\right\rangle,
\end{align*}
where $Y\in\mathcal{O}$ and $U,U'\in\mathfrak{g}$. 
Recall that the tangent space $T_Y\mathcal{O}$ to the orbit $\mathcal{O}$ is given as $T_Y\mathcal{O}=\left\{\left.\left(\mathrm{ad}_U\right)^{\mathrm{T}}Y\right| U\in\mathfrak{g}\right\}$.  
By the generalities of the symplectic stratification of a Poisson manifold, the Hamiltonian vector field $\Xi_h$ for the Hamiltonian $h\in \mathcal{C}^{\infty}(\mathfrak{g})$ can be restricted to the leaf $\mathcal{O}$
and the restricted vector field on $\mathcal{O}$ is also Hamiltonian with respect to the orbit symplectic form $\omega$ and the Hamiltonian function $h|_{\mathcal{O}}$. 

The restriction of the Lie-Poisson equation to the symplectic manifold $\left(\mathcal{O}, \omega\right)$ coincides also with the reduced Hamilton equation through Marsden-Weinstein Reduction. 
On $T^{\ast}G\cong G\times \mathfrak{g}^{\ast}$, we consider the $\Ad^{\ast}$-equivariant momentum mapping $J: T^{\ast}G\cong G\times \mathfrak{g}^{\ast}\ni (q,\mu)\mapsto \Ad_{q^{-1}}^{\ast}\mu\in \mathfrak{g}^{\ast}$. 
Through the identification $(G\times \mathfrak{g}\cong )TG\cong T^{\ast}G$ via the (pseudo-)Riemannian metric $\langle \cdot, \cdot\rangle$, the momentum mapping is written as $J: TG\cong G\times \mathfrak{g}\ni (q,Y)\mapsto \left(\Ad_{q^{-1}}\right)^{\mathrm{T}}Y\in \mathfrak{g}$. 
We take an element $Y_0\in \mathcal{O}\subset \mathfrak{g}$ and then the reduced phase space is given as $G_{Y_0}\backslash J^{-1}(Y_0)\cong \mathcal{O}$, where 
$G_{Y_0}:=\left\{q\in G\left| \left(\Ad_q\right)^{\mathrm{T}}(Y_0)=Y_0\right.\right\}$ is the isotropy subgroup.
The Hamiltonian system $\left(\mathcal{O}, \omega, h|_{\mathcal{O}}\right)$ coincides with the reduced system from $\left( TG, \Omega, H \right)$. 

Through Marsden-Weinstein Reduction Theorem, an equilibrium of the Lie-Poisson equation restricted on the symplectic leaf $\mathcal{O}$ can be regarded as a {\it relative equilibrium} of the original Hamiltonian system $\left( TG, \Omega, H \right)$. 
See e.g. \cite{arnold_1989} for the notion of relative equilibria. 

\subsection{Lie-Poisson equation for step-two nilpotent Lie groups}\label{Lie-Poisson equation for step-two nilpotent Lie groups}\label{subsec_lp}

We apply the general framework of the left-invariant geodesic flows of Lie groups disscused in the previous section to step-two nilpotent Lie groups.

Let $G$ be a Lie group. The corresponding Lie algebra is denoted by $\mathfrak{g}$ as in the previous subsection. We assume that the Lie group $G$ is step-two nilpotent in the following sense:
\begin{definition}
The Lie group $G$ is called {\it{step-two nilpotent}} if the corresponding Lie algebra $\mathfrak{g}$ is step-two nilpotent, that is, $\mathfrak{g}$ satisfies $\left[\mathfrak{g},\mathfrak{g}\right]\subset\mathfrak{z}$, where $\mathfrak{z}$ is the centre of $\mathfrak{g}$.
\end{definition}

We take a scalar product $\langle \cdot, \cdot\rangle$ on $\mathfrak{g}$ and assume that its restriction $\langle \cdot, \cdot\rangle_{\mathfrak{z}}:=\langle \cdot, \cdot\rangle|_{\mathfrak{z}\times\mathfrak{z}}$ to $\mathfrak{z}$ is non-degenerate and hence again a scalar product on $\mathfrak{z}$. 
Then, setting $\mathfrak{v}:=\left\{V\in \mathfrak{g}\mid \forall Z\in\mathfrak{z},\; \langle V, Z\rangle=0\right\}$, we have $\mathfrak{g}=\mathfrak{v}\dot{+}\mathfrak{z}$ and $\left[\mathfrak{v}, \mathfrak{v}\right]\subset \mathfrak{z}$. 
(See \cite[Lemma 23, p.49]{oneill_1983}, following which $\mathfrak{v}$ should be called the orthogonal perp of $\mathfrak{z}$ and denoted as $\mathfrak{v}=\mathfrak{z}^{\perp}$.) 
Note that the restriction $\langle \cdot, \cdot \rangle_{\mathfrak{v}}:=\langle \cdot, \cdot \rangle|_{\mathfrak{v}\times\mathfrak{v}}$ of the scalar product $\langle \cdot, \cdot \rangle$ to ${\mathfrak{v}}$ is also a non-degenerate scalar product since $\mathfrak{v}^{\perp}=\left(\mathfrak{z}^{\perp}\right)^{\perp}=\mathfrak{z}$. 
(See \cite[Lemma 22, p.49]{oneill_1983}.)
\begin{remark}
A Lie algebra is said to be {\it{step-$r$ graded}} if it is decomposed into a direct sum of vector subspaces $\mathfrak{g}_{1}\dot{+}\cdots\dot{+}\mathfrak{g}_{r}$ satisfying $\left[\mathfrak{g}_{i},\mathfrak{g}_{j}\right]\subset \mathfrak{g}_{i+j}$ for all $i,j\in\left\{1, \ldots, r\right\}$, and $\mathfrak{g}_{s}=\left\{0\right\}$ for all $s>r$. 
(See e.g \cite[p.18]{montgomery_2002}.) For a step-two nilpotent Lie algebra $\mathfrak{g}$ equipped with a scalar product, we have $\mathfrak{g}=\mathfrak{g}_{1}\dot{+}\mathfrak{g}_{2}$ and $\left[\mathfrak{g}_{1},\mathfrak{g}_{1}\right]\subset \mathfrak{g}_{2}$, where $\mathfrak{g}_{1}\coloneq\mathfrak{v},\,\mathfrak{g}_{2}\coloneq \mathfrak{z}$. Hence, the step-two nilpotent Lie algebra $\mathfrak{g}$ equipped with a scalar product is step-two graded.
\end{remark}

Now, we consider the linear operator $j: \mathfrak{z}\rightarrow \mathrm{End}\left(\mathfrak{v}\right)$, called as ``$j$-mapping,'' defined through 
\[
\left\langle \left[ V, V^{\prime} \right],Z\right\rangle_{\mathfrak{z}}=\left\langle j\left( Z \right)V,V^{\prime}\right\rangle_{\mathfrak{v}}, \qquad V, V^{\prime}\in \mathfrak{v}, Z\in\mathfrak{z}. 
\]
See \cite{eberlein_1994} for more information around the ``$j$-mapping.'' 
By the skew-symmetry of the Lie bracket, it is obvious that, for any $Z\in \mathfrak{z}$, $j(Z)\in \mathrm{End}\left(\mathfrak{v}\right)$ is skew-symmetric with respect to the scalar product $\langle \cdot, \cdot\rangle_{\mathfrak{v}}$, namely the image of the mapping $j$ is in the Lie algebra of skew-symmetric linear endomorphisms of $\mathfrak{v}$ with respect to the scalar product $\langle \cdot, \cdot\rangle_{\mathfrak{v}}$: $j\left(\mathfrak{z}\right)\subset \mathfrak{so}\left( \mathfrak{v},\langle\cdot,\cdot\rangle_{\mathfrak{v}} \right)$. 

\medskip

In terms of the mapping $j: \mathfrak{z}\rightarrow \mathfrak{so}\left(\mathfrak{v}, \langle \cdot, \cdot\rangle_{\mathfrak{v}}\right)$, we more explicitly describe the Lie-Poisson equation for the Hamiltonian system associated to a left-invariant Hamiltonian function on the (co)tangent bundle of the step-two nilpotent Lie group $G$ with respect to the left-invariant pseudo-Riemannian metric associated to the scalar product $\langle\cdot, \cdot \rangle$. 
With respect to the direct sum decomposition $\mathfrak{g}=\mathfrak{v}\dot{+}\mathfrak{z}$, an element $Y\in \mathfrak{g}$ of the Lie algebra is uniquely written as $Y=Y_{\mathfrak{v}}+Y_{\mathfrak{z}}$, where $Y_{\mathfrak{v}}\in\mathfrak{v}$, $Y_{\mathfrak{z}}\in\mathfrak{z}$. 
For any $f, f^{\prime}\in \mathcal{C}^{\infty}\left(\mathfrak{g}\right)$, the Lie-Poisson bracket is given as 
\begin{align*}
  \left\{ f, f^{\prime} \right\}\left( Y \right) 
  = 
  -\left\langle Y_{\mathfrak{z}}, \left[ V_{\mathfrak{v}}, V_{\mathfrak{v}}^{\prime} \right]\right\rangle
  =
  -\left\langle j\left(Y_{\mathfrak{z}}\right)V_{\mathfrak{v}}, V_{\mathfrak{v}}^{\prime} \right\rangle, 
\end{align*} 
where $V=V_{\mathfrak{v}}+V_{\mathfrak{z}}=\mathrm{grad}_{Y}f, \,V^{\prime}=V_{\mathfrak{v}}^{\prime}+V_{\mathfrak{z}}^{\prime}=\mathrm{grad}_{Y}f'\in \mathfrak{g}=\mathfrak{v}\dot{+}\mathfrak{z}$. 
It should be pointed out that if a function $f=f(Y)$, $Y\in\mathfrak{g}$, on the Lie algebra only depends on the centre component $Y_{\mathfrak{z}}\in\mathfrak{z}$ of $Y$, then it is a Casimir function: $\left\{f, \cdot\right\}=0$. 

For any $U\in \mathfrak{g}$, the coadjoint mapping $\left(\mathrm{ad}_U\right)^{\mathrm{T}}: \mathfrak{g}\rightarrow \mathfrak{g}$ defined through \eqref{eq_adT} can explicitly be given as 
\begin{align}\label{adjoint rep}
\left(\mathrm{ad}_U\right)^{\mathrm{T}}(Y)=j\left( Y_{\mathfrak{z}}\right)U_{\mathfrak{v}}, \qquad Y\in \mathfrak{g}\cong\mathfrak{g}^{\ast}. 
\end{align}
If the step-two nilpotent Lie group $G$ is connected, we can concretely describe the coadjoint orbit $\mathcal{O}\subset\mathfrak{g}$ through $Y\in\mathfrak{g}$ as in the following proposition:
\begin{proposition}\label{coadjoint orbit of step-two}
If the Lie group $G$ is connected step-two nilpotent, the coadjoint orbit $\mathcal{O}\subset\mathfrak{g}$ through $Y\in\mathfrak{g}$ is given as
\[
\mathcal{O} = Y + \mathrm{Im}\left(j\left(Y_{\mathfrak{z}}\right)\right)\coloneq\Set{Y+j\left(Y_{\mathfrak{z}}\right)U_{\mathfrak{v}} | U\in\mathfrak{g}}.
\]
In particular, the coadjoint orbit $\mathcal{O}$ is a linear manifold.
\hfill $\square$
\end{proposition}
\begin{proof}
It is well-known that, if the nilpotent Lie group $G$ is connected, the exponential map $\exp_{G}\colon\mathfrak{g}\to G$ of $G$ is surjective. (See e.g. \cite[Corollary 11.2.7, p.446]{hilgert-neeb_2012}.) 
Recall the well-known formula $\Ad_{\exp_{G} U}=\exp\left(\ad_{U}\right)$ for $U\in\mathfrak{g}$. (See e.g. \cite[\S 9.2.3]{hilgert-neeb_2012}.) Here, by $\exp\colon \mathrm{End}\left(\mathfrak{g}\right)\to GL\left(\mathfrak{g}\right)$, we denote the matrix exponential function: $\displaystyle\exp\left(A\right)=\sum_{n=0}^{\infty}\frac{A^{n}}{n!},\, A\in\mathrm{End}(\mathfrak{g})$. Using it, we have
\begin{align*}
\left\langle \left(\Ad_{\exp_{G} U}\right)^{\mathrm{T}}Y, X\right\rangle&=\left\langle Y,\left(\Ad_{\exp_{G} U}\right)X \right\rangle
=\left\langle Y, \exp\left(\ad_{U}\right) X\right\rangle=\left\langle \exp\left(\left(\ad_{U}\right)^{\mathrm{T}}\right)Y, X\right\rangle,
\end{align*}
where $X,Y,U\in\mathfrak{g}$, and hence we obtain $\left(\Ad_{\exp_{G} U}\right)^{\mathrm{T}}=\exp\left(\left(\ad_{U}\right)^{\mathrm{T}}\right)$. 
Since $\mathfrak{g}$ is a step-two nilpotent Lie algebra, we have $\left(\mathrm{ad}_U\right)^2=0$, which yields $\exp\left(\left(\ad_{U}\right)^{\mathrm{T}}\right)=\mathrm{id}_{\mathfrak{g}}+\left(\ad_{U}\right)^{\mathrm{T}}$. 
Consequently, we have $\left(\Ad_{\exp_{G} U}\right)^{\mathrm{T}}=\mathrm{id}_{\mathfrak{g}}+\left(\ad_{U}\right)^{\mathrm{T}}$, and the coadjoint orbit $\mathcal{O}$ through $Y\in\mathfrak{g}$ is given as 
\[
\mathcal{O}=\Set{\left(\Ad_{\exp_{G} U}\right)^{\mathrm{T}}Y | U\in\mathfrak{g}}=\Set{Y+\left(\ad_{U}\right)^{\mathrm{T}}Y | U\in\mathfrak{g}}=\Set{Y+j\left(Y_{\mathfrak{z}}\right)U_{\mathfrak{v}} | U\in\mathfrak{g}},
\]
by the surjectivity of the exponential map and the formula \eqref{adjoint rep}.
\end{proof}
In particular, the tangent space $T_Y\mathcal{O}$ to the coadjoint orbit $\mathcal{O}\subset \mathfrak{g}$ through $Y\in\mathfrak{g}$ is given as 
\[
T_Y\mathcal{O}=\mathrm{Im}\left(j\left( Y_{\mathfrak{z}}\right)\right)=\left(j\left( Y_{\mathfrak{z}}\right)\right)\left(\mathfrak{v}\right). 
\]
At the point $Y\in \mathcal{O}$, the orbit symplectic form $\omega$ is given as 
\begin{align*}
  \omega_{Y}\left( \left( \left( \ad_{U} \right)^{\mathrm{T}}Y \right),\left( \left(\ad_{U'} \right)^{\mathrm{T}}Y \right) \right)
  =
  -\left\langle Y_{\mathfrak{z}}, \left[ U_{\mathfrak{v}},U'_{\mathfrak{v}} \right] \right\rangle,
\end{align*}
where $U=U_{\mathfrak{v}}+U_{\mathfrak{z}},\,U'=U'_{\mathfrak{v}}+U'_{\mathfrak{z}}\in\mathfrak{g}=\mathfrak{v}\dot{+}\mathfrak{z}$. 

For a left-invariant Hamiltonian $H\in\mathcal{C}^\infty\left( TG \right)$, the reduced Hamiltonian function on $\mathfrak{g}$ is the restriction $h=H\left|_{\left\{ e \right\}\times\mathfrak{g}}\right.$ to $\left\{ e \right\}\times\mathfrak{g}\subset G\times\mathfrak{g}\cong TG$. 
By using the linear map $j\colon\mathfrak{z}\to \mathfrak{so}\left( \mathfrak{v},\langle\cdot,\cdot\rangle_{\mathfrak{v}} \right)$, the Hamiltonian vector field $\Xi_{h}$ for $h\in\mathcal{C}^\infty\left( \mathfrak{g} \right)$ is written as 
\begin{align*}
    \left( \Xi_{h} \right)_Y =j\left( Y_{\mathfrak{z}} \right)V_{\mathfrak{v}},
\end{align*}
where $V=V_{\mathfrak{v}}+V_{\mathfrak{z}}=\mathrm{grad}_{Y}h\in \mathfrak{g}=\mathfrak{v}\dot{+}\mathfrak{z}$, and hence the Lie-Poisson equation is described as
\begin{align}
  \begin{cases}
    \displaystyle\frac{\dd Y_{\mathfrak{v}}}{\dd t}=j\left( Y_{\mathfrak{z}} \right)V_{\mathfrak{v}},
    \\[7pt]
    \displaystyle\frac{\dd Y_{\mathfrak{z}}}{\dd t}=0. \label{reduced_hamilton_eq}
  \end{cases}
\end{align}
In particular, the reduced Hamiltonian function for the geodesic flow with respect to the left-invariant pseudo-Riemannian metric is defined through $\displaystyle h\left( Y \right)=\frac{1}{2}\left\langle Y,Y\right\rangle$ and the Hamilton equation \eqref{reduced_hamilton_eq} is given as
\begin{align}
  \begin{cases}
    \displaystyle\frac{\dd Y_{\mathfrak{v}}}{\dd t}=j\left( Y_{\mathfrak{z}} \right)Y_{\mathfrak{v}},
    \\[7pt]
    \displaystyle\frac{\dd Y_{\mathfrak{z}}}{\dd t}=0. \label{eq_geodesic_flow}
  \end{cases}
\end{align}

The Lie-Poisson equation \eqref{eq_geodesic_flow} can be restricted to the coadjoint orbit $\mathcal{O}$ and it is nothing but the reduced system obtained from the geodesic flow through Marsden-Weinstein Reduction. 
In the subsequent section, we analyse the stability of the equilibrium points for the equation \eqref{eq_geodesic_flow}. 
By the stability of equilibria, we mean the property about the restricted system of \eqref{eq_geodesic_flow} to a generic, i.e. maximal-dimensional, orbit $\mathcal{O}$. 
Before giving the detailed account on the analysis, we review certain classes of step-two nilpotent Lie groups in the next subsection. 

\begin{remark}\label{remark on degenerate case}
Fixing the (pseudo-)Riemannian metric $\langle \cdot, \cdot\rangle$, we can think of more general left-invariant metrics $g=\left(g_q\right)_{q\in G}$ for which the restriction of $g_e:\mathfrak{g}\times \mathfrak{g}\rightarrow \mathbb{R}$ to $\mathfrak{z}\times \mathfrak{z}$ needs not be non-degenerate. 
In fact, a left-invariant (pseudo-)Riemannian metric $g$ can be described as 
\[
g_e\left(U, V\right)
=
\left\langle U_{\mathfrak{v}}, g_{11}V_{\mathfrak{v}} + g_{12}V_{\mathfrak{z}}\right\rangle_{\mathfrak{v}}
+
\left\langle U_{\mathfrak{z}}, g_{21}V_{\mathfrak{v}} + g_{22}V_{\mathfrak{z}}\right\rangle_{\mathfrak{z}},
\]
for all $U=U_{\mathfrak{v}}+U_{\mathfrak{z}}, V=V_{\mathfrak{v}}+V_{\mathfrak{z}} \in \mathfrak{g}=\mathfrak{v}\dot{+}\mathfrak{z}$. 
Here, $g_{11}: \mathfrak{v}\rightarrow\mathfrak{v}$, $g_{12}: \mathfrak{z}\rightarrow\mathfrak{v}$, $g_{21}: \mathfrak{v}\rightarrow\mathfrak{z}$, $g_{22}: \mathfrak{z}\rightarrow\mathfrak{z}$ are linear and satisfy the relations $g_{11}^{\mathrm{T}}=g_{11}$, $g_{12}^{\mathrm{T}}=g_{21}$, $g_{21}^{\mathrm{T}}=g_{12}$, $g_{22}^{\mathrm{T}}=g_{22}$.  
Namely, 
\begin{align*}
\left\langle V, g_{11}V^{\prime}\right\rangle_{\mathfrak{v}}
=
\left\langle g_{11}V, V^{\prime}\right\rangle_{\mathfrak{v}}, \quad 
\left\langle V, g_{12}Z\right\rangle_{\mathfrak{v}}
=
\left\langle g_{21}V, Z\right\rangle_{\mathfrak{z}}, \quad 
\left\langle Z, g_{22}Z^{\prime}\right\rangle_{\mathfrak{z}}
=
\left\langle g_{22}Z, Z^{\prime}\right\rangle_{\mathfrak{z}},
\end{align*}
for all $V, V^{\prime}\in\mathfrak{v}$, $Z, Z^{\prime}\in\mathfrak{z}$. 
We denote the inverse of the mapping 
\[
\mathfrak{v}\oplus\mathfrak{z}\ni \left(V, Z\right)\mapsto \left(g_{11}(V)+g_{12}(Z), g_{21}(V)+g_{22}(Z)\right)\in\mathfrak{v}\oplus\mathfrak{z}
\]
by  
\[
\mathfrak{v}\oplus\mathfrak{z}\ni \left(V, Z\right)\mapsto \left(g^{11}(V)+g^{12}(Z), g^{21}(V)+g^{22}(Z)\right)\in\mathfrak{v}\oplus\mathfrak{z}. 
\]
Then, the linear mappings $g^{11}: \mathfrak{v}\to\mathfrak{v}$, $g^{12}: \mathfrak{z}\to\mathfrak{v}$, $g^{21}: \mathfrak{v}\to\mathfrak{z}$, $g^{22}: \mathfrak{z}\to\mathfrak{z}$ satisfy $\left(g^{ij}\right)^{\mathrm{T}}=g^{ji}$. Namely,
\begin{align*}
\left\langle V, g^{11}V^{\prime}\right\rangle_{\mathfrak{v}}
=
\left\langle g^{11}V, V^{\prime}\right\rangle_{\mathfrak{v}}, \quad 
\left\langle V, g^{12}Z\right\rangle_{\mathfrak{v}}
=
\left\langle g^{21}V, Z\right\rangle_{\mathfrak{z}}, \quad 
\left\langle Z, g^{22}Z^{\prime}\right\rangle_{\mathfrak{z}}
=
\left\langle g^{22}Z, Z^{\prime}\right\rangle_{\mathfrak{z}},
\end{align*}
for all $V, V^{\prime}\in\mathfrak{v}$, $Z, Z^{\prime}\in\mathfrak{z}$. 
Using the Legendre transform, the Hamiltonian for the geodesic flow with respect to the metric $g$ is given by the function 
\[
H^g\left(Y\right)
=
\dfrac{1}{2}\left(\left\langle Y_{\mathfrak{v}}, g^{11}Y_{\mathfrak{v}}+g^{12}Y_{\mathfrak{z}}\right\rangle_{\mathfrak{v}}+\left\langle Y_{\mathfrak{z}}, g^{21}Y_{\mathfrak{v}}+g^{22}Y_{\mathfrak{z}}\right\rangle_{\mathfrak{z}}\right), 
\quad Y=Y_{\mathfrak{v}}+Y_{\mathfrak{z}}\in\mathfrak{v}\dot{+}\mathfrak{z}, 
\]
on the Lie algebra $\mathfrak{g}$. 
Since $\mathrm{grad}H^g(Y)=g^{11}Y_{\mathfrak{v}}+g^{12}Y_{\mathfrak{z}}+ g^{21}Y_{\mathfrak{v}}+g^{22}Y_{\mathfrak{z}}$, the Lie-Poisson equation \eqref{reduced_hamilton_eq} is now written as 
\[
 \begin{cases}
    \displaystyle\frac{\dd Y_{\mathfrak{v}}}{\dd t}=j\left( Y_{\mathfrak{z}} \right)\left(g^{11}Y_{\mathfrak{v}}+g^{12}Y_{\mathfrak{z}}\right),
    \\[7pt]
    \displaystyle\frac{\dd Y_{\mathfrak{z}}}{\dd t}=0. 
  \end{cases}
\]
However, we do not discuss further properties of this general geodesic flow in the present paper.

\hfill $\square$
\end{remark}

\subsection{Certain classes of step-two nilpotent Lie groups}\label{Certain classes of step-two nilpotent Lie groups}
In this section, we review certain classes of step-two nilpotent Lie groups/algebras which appear in the literature. 

\subsubsection{Heisenberg Lie groups}\label{subsec_heis}
As an important example, we first review the Heisenberg Lie groups. 
See e.g. \cite[Chapter 1]{taylor_1986} for the details. 
For $\left(x, y, z\right), \left(x^{\prime}, y^{\prime}, z^{\prime}\right)\in \mathbb{R}^n\times \mathbb{R}^n\times\mathbb{R}$, the operation 
\[
\left(x, y, z\right)\ast\left(x^{\prime}, y^{\prime}, z^{\prime}\right)
:=
\left(x+x^{\prime}, y+y^{\prime}, z+z^{\prime}+\dfrac{x^{\mathrm{T}}y^{\prime}-y^{\mathrm{T}}x^{\prime}}{2}\right), 
\]
induces a Lie group structure on $\mathbb{R}^n\times \mathbb{R}^n\times \mathbb{R}$ and the group $\mathbb{H}_{2n+1}:=\left(\mathbb{R}^n\times \mathbb{R}^n\times \mathbb{R}, \ast\right)$ is called the $2n+1$-dimensional {\it Heisenberg Lie group}. 
The associated Lie algebra $\mathfrak{h}_{2n+1}$ is described as $\mathbb{R}^n\times \mathbb{R}^n\times \mathbb{R}$ equipped with the Lie bracket 
\[
\left[\left(\xi, \eta,\zeta\right), \left(\xi^{\prime}, \eta^{\prime},\zeta^{\prime}\right)\right]
=\left(0, 0, \xi^{\mathrm{T}}\eta^{\prime}-\eta^{\mathrm{T}}\xi^{\prime}\right) 
\]
for $\left(\xi, \eta,\zeta\right), \left(\xi^{\prime}, \eta^{\prime},\zeta^{\prime}\right)\in \mathbb{R}^n\times \mathbb{R}^n\times \mathbb{R}$. 
We consider the inner product 
\[
\left\langle\left(\xi, \eta,\zeta\right), \left(\xi^{\prime}, \eta^{\prime},\zeta^{\prime}\right)\right\rangle
=
\xi^{\mathrm{T}}\xi^{\prime}+\eta^{\mathrm{T}}\eta^{\prime}+\zeta\zeta^{\prime}, 
\quad \left(\xi, \eta,\zeta\right), \left(\xi^{\prime}, \eta^{\prime},\zeta^{\prime}\right)\in \mathbb{R}^n\times \mathbb{R}^n\times \mathbb{R}, 
\]
on $\mathfrak{h}_{2n+1}$ which extends to the left-invariant metric on $\mathbb{H}_{2n+1}$ through the left-translations. 
The centre $\mathfrak{z}$ of $\mathfrak{h}_{2n+1}$ is given as 
\[
\mathfrak{z}=\left\{\left(0,0,\zeta\right)\mid \zeta\in \mathbb{R}\right\}
\]
and its orthogonal complement is given as 
\[
\mathfrak{v}=\left\{\left(\xi,\eta,0\right)\mid \xi, \eta\in\mathbb{R}^n\right\}. 
\]
By definition, the $j$-mapping $J_{\zeta}: \mathfrak{v}\rightarrow \mathfrak{v}$ satisfies 
\[
\left\langle J_{\zeta}\left(\xi,\eta\right), \left(\xi^{\prime}, \eta^{\prime}\right)\right\rangle_{\mathfrak{v}}
=
\left\langle \left(0, 0, \zeta\right), \left[\left(\xi,\eta, 0\right), \left(\xi^{\prime}, \eta^{\prime}, 0\right)\right]\right\rangle
=
\left(\xi^{\mathrm{T}}, \eta^{\mathrm{T}}\right)\begin{pmatrix}0 & \zeta \mathsf{E}_n \\ -\zeta \mathsf{E}_n  & 0\end{pmatrix}\begin{pmatrix}\xi^{\prime} \\ \eta^{\prime}\end{pmatrix}, 
\]
where $\xi,\eta, \xi^{\prime}, \eta^{\prime}\in\mathbb{R}^n$, $\zeta, \zeta^{\prime}\in \mathbb{R}$. 
Thus, we have 
\[
J_{\zeta}
=
\begin{pmatrix}0 & -\zeta \mathsf{E}_n \\ \zeta \mathsf{E}_n  & 0\end{pmatrix}, \qquad \zeta\in \mathbb{R}. 
\]
In what follows, we mention several known classes of step-two nilpotent Lie groups, each of which is modeled after the Heisenberg Lie groups. 

\subsubsection{Step-two Carnot groups}\label{subsubsubsection_2_3_1}
A nilpotent Lie group is called a step-two {\it Carnot group} (or stratified group) if the corresponding Lie algebra $\mathfrak{g}$ is decomposed as $\mathfrak{g}=\mathfrak{v}\dot{+}\mathfrak{z}$ into the direct sum of the centre $\mathfrak{z}$ of $\mathfrak{g}$ and its complement $\mathfrak{v}$, which satisfy $\left[\mathfrak{v}, \mathfrak{v}\right]=\mathfrak{z}$. 
See e.g. \cite[\S 1.10]{montgomery_2002} for the details on general Carnot groups. 
Note that the Lie algebras of step-two Carnot groups are also called as non-degenerate step-two nilpotent Lie algebras. 
See \cite{borovoi_dina_degraaf_2021}. 
However, we do not use this terminology to avoid possible confusions. 

In view of the definition of the ``$j$-mapping,'' a step-two nilpotent Lie algebra $\mathfrak{g}$ equipped with a scalar product $\left\langle\cdot, \cdot \right\rangle$ which allows the orthogonal direct sum decomposition $\mathfrak{g}=\mathfrak{v}\dot{+}\mathfrak{z}$ into the centre $\mathfrak{z}$ and its complement $\mathfrak{v}$, the Lie algebra $\mathfrak{g}$ is Carnot if and only if $j\left( Z \right)\colon\mathfrak{v}\to\mathfrak{v}$ is surjective for any $Z\in\mathfrak{z}\setminus\left\{ 0 \right\}$. 
In particular, the dimension of the complement $\mathfrak{v}$ is necessarily even.

\paragraph{M\'{e}tivier groups}
As an important subclass of step-two Carnot groups, we consider the M\'{e}tivier groups. 
A Lie group $G$ is called a \textit{M\'{e}tivier group} (cf. \cite{metivier_1980}) if the corresponding Lie algebra $\mathfrak{g}$ has the (vector space) direct sum decomposition $\mathfrak{g}=\mathfrak{v}\dot{+}\mathfrak{z}$, where $\mathfrak{z}$ is the centre and $\mathfrak{v}$ is its complement and if for all non-zero linear functional $\eta\in \mathfrak{z}^{\ast}\setminus \{0\}$, the bilinear form $\mathfrak{v}\times \mathfrak{v}\ni (V,V^{\prime})\mapsto \eta\left(\left[V, V^{\prime}\right]\right)\in \mathbb{R}$ is non-degenerate. 

\begin{proposition}
A M\'{e}tivier group is a step-two Carnot group. 
\hfill $\square$
\end{proposition}
\begin{proof}
It suffices to prove the proposition at the level of Lie algebras. 
Let $\mathfrak{g}$ be a Lie algebra whose corresponding Lie group is a M\'{e}tivier group. 
Then, we have the direct sum decomposition $\mathfrak{g}=\mathfrak{v}\dot{+}\mathfrak{z}$, where $\mathfrak{z}$ is the centre of $\mathfrak{g}$. 
By definition, the bilinear form 
\[
\eta\left(\left[V, V^{\prime}\right]\right), \qquad V, V^{\prime}\in \mathfrak{v}, 
\]
is non-degenerate for all $\eta\in \mathfrak{z}^{\ast}\setminus \{0\}$. 
If $\mathfrak{g}$ would not be Carnot, we would have $\left[\mathfrak{v}, \mathfrak{v}\right]\subsetneqq \mathfrak{z}$, namely $\exists Z\in \mathfrak{z}$ such that $Z\not\in \left[\mathfrak{v}, \mathfrak{v}\right]$. 
Then, there would be a linear functional $\eta\in \mathfrak{z}^{\ast}$ such that $\eta\left(\left[\mathfrak{v}, \mathfrak{v}\right]\right)=0$ and $\eta(Z)\neq 0$. 
In this case, the bilinear form $\eta\left(\left[\cdot, \cdot\right]\right): \mathfrak{v}\times \mathfrak{v}\rightarrow \mathbb{R}$ would clearly be zero and hence degenerate. 
This is a contradiction. 
\end{proof}

\begin{example}
We consider the direct product $\mathbb{H}_3\times \mathbb{H}_3$ of two copies of the three-dimensional Heisenberg Lie group $\mathbb{H}_3$ as an example of Carnot groups which are not M\'{e}tivier groups. 
To see this, we consider the generators $X, Y, Z$ of the three-dimensional Heisenberg Lie algebra $\mathfrak{h}_3$ which satisfy the commutation relations $\left[X, Y\right]=Z$, $\left[Y, Z\right]=0$, $\left[Z, X\right]=0$. 
Clearly $\mathrm{span}_{\mathbb{R}}\left\{Z\right\}$ is the centre of $\mathfrak{h}_3$. 

Next, the direct sum Lie algebra $\mathfrak{h}_3\oplus \mathfrak{h}_3$ is generated by $X_i, Y_i, Z_i$ satisfying $\left[X_i, Y_j\right]=\delta_{ij}Z_i$, $\left[Y_i, Z_j\right]=0$, $\left[Z_i, X_j\right]=0$ for $i, j=1,2$. 
The centre of the Lie algebra $\mathfrak{h}_3\oplus \mathfrak{h}_3$ is given as $\mathfrak{z}:=\mathrm{span}_{\mathbb{R}}\left\{Z_1, Z_2\right\}$ and we take its complement $\mathfrak{v}:=\mathrm{span}_{\mathbb{R}}\left\{X_1, X_2, Y_1, Y_2\right\}$. 
By the commutation relations, we clearly have $\left[\mathfrak{v}, \mathfrak{v}\right]=\mathfrak{z}$ and hence $\mathfrak{h}_3\oplus \mathfrak{h}_3$ is a step-two Carnot algebra. 

On the other hand,  $\mathfrak{h}_3\oplus \mathfrak{h}_3$ is not a M\'{e}tivier algebra. 
In fact, we take $\eta\in\mathfrak{z}^{\ast}$ defined through $\eta(Z_1)=1$, $\eta(Z_2)=0$ and then the skew-symmetric bilinear form $\eta\left(\left[\cdot, \cdot\right]\right): \mathfrak{v}\times \mathfrak{v}\rightarrow \mathbb{R}$ is degenerate, since $\eta\left([X_2, X_1]\right)=\eta\left([X_2, X_2]\right)=\eta\left([X_2, Y_1]\right)=\eta\left([X_2, Y_2]\right)=0$. 
\hfill $\blacklozenge$
\end{example}

Note that a step-two nilpotent Lie algebra $\mathfrak{g}$ equipped with a scalar product $\left\langle\cdot, \cdot \right\rangle$ with the orthogonal direct sum decomposition $\mathfrak{g}=\mathfrak{v}\dot{+}\mathfrak{z}$, where $\mathfrak{z}$ is the centre, the Lie algebra $\mathfrak{g}$ is M\'{e}tivier if and only if the scalar product is positive-definite and $j(Z):\mathfrak{v}\rightarrow \mathfrak{v}$ is surjective for all $Z\in \mathfrak{z}\setminus\left\{ 0 \right\}$. 

\paragraph{$H$-type and pseudo-$H$-type Lie groups}
We here consider the $H$-type Lie groups introduced by Kaplan \cite{kaplan_1980}, as well as their extension, pseudo-$H$-type Lie groups, considered by Furutani and Markina \cite{furutani_markina_2014,furutani_markina_2017,furutani_markina_2019}. 

We take a step-two nilpotent Lie group $G$ whose Lie algebra is denoted by $\mathfrak{g}$ which is supposed to be equipped with a scalar product $\left\langle \cdot, \cdot\right\rangle: \mathfrak{g}\times \mathfrak{g}\rightarrow \mathbb{R}$. 
As in the previous section, we further assume that the restriction $\left\langle \cdot, \cdot \right\rangle_{\mathfrak{z}}:=\left\langle \cdot, \cdot \right\rangle|_{\mathfrak{z}\times \mathfrak{z}}$ is non-degenerate. 
In this case, the orthogonal complement $\mathfrak{v}$ also has the induced scalar product $\left\langle \cdot, \cdot \right\rangle_{\mathfrak{v}}:=\left\langle \cdot, \cdot \right\rangle|_{\mathfrak{v}\times \mathfrak{v}}$. 
(cf. \cite[Lemma 22, p.49]{oneill_1983}.) 
We consider the mapping $j(Z): \mathfrak{v}\rightarrow \mathfrak{v}$ for all $Z\in\mathfrak{z}$ defined through $\left\langle j(Z)V, V^{\prime}\right\rangle_{\mathfrak{v}}=\left\langle Z, \left[V, V^{\prime}\right]\right\rangle_{\mathfrak{z}}$ where $V, V^{\prime}\in\mathfrak{v}$. 
The Lie group is called a \textit{pseudo-$H$-type Lie group}, if $j(Z)$ satisfies 
\begin{equation}\label{relation_pseudo-h-type}
\left\langle j(Z)V, j(Z)V^{\prime}\right\rangle_{\mathfrak{v}}=\left\langle Z, Z\right\rangle_{\mathfrak{z}}\cdot \left\langle V, V^{\prime}\right\rangle_{\mathfrak{v}}, \qquad V, V^{\prime}\in \mathfrak{v}, \, Z\in\mathfrak{z}. 
\end{equation}
In the case where the scalar product $\left\langle\cdot, \cdot\right\rangle$ is positive-definite and hence it is an inner product, the restriction $\left\langle\cdot, \cdot\right\rangle_{\mathfrak{z}}$ is an inner product and automatically non-degenerate. 
In this case, the step-two nilpotent Lie group for which \eqref{relation_pseudo-h-type} is satisfied is called an \textit{$H$-type Lie group}.  

Note that an $H$-type Lie group is automatically a M\'{e}tivier group. 
On the other hand, a M\'{e}tivier group is not necessarily an $H$-type Lie group. 

\begin{example}
Following \cite[p.176 Remark 3.7.5]{bonfiglioli_lanconelli_uguzzoni_2007}, we consider the Lie algebra $\mathfrak{g}\cong\mathbb{R}^5$ generated by $X_1, X_2, Y_1, Y_2, Z$ with the commutation relations $\left[X_1, Y_1\right]=Z$, $\left[X_2, Y_2\right]=2Z$, $\left[X_1, X_2\right]=\left[X_1, Y_2\right]=\left[X_2, Y_1\right]=\left[Y_1, Y_2\right]=\left[Z, X_1\right]=\left[Z, X_2\right]=\left[Z, Y_1\right]=\left[Z, Y_2\right]=0$. 
We consider the inner product on $\mathfrak{g}$ for which $X_1, X_2, Y_1, Y_2, Z$ are orthonormal. 
Clearly, the centre is given as $\mathfrak{z}=\mathrm{span}_{\mathbb{R}}\left\{Z\right\}$ and $\mathfrak{v}=\mathrm{span}_{\mathbb{R}}\left\{X_1, X_2, X_3, X_4\right\}$ is its orthogonal complement. 
Then, with respect to the basis $X_1, Y_1, X_2, Y_2$, the ``$j$-mapping'' is given as 
\[
j(Z)
=
\begin{pmatrix}
0 & -1 & 0 & 0 \\
1 & 0 & 0 & 0 \\
0 & 0 & 0 & -2 \\
0 & 0 & 2 & 0 
\end{pmatrix}. 
\]
The Lie algebra $\mathfrak{g}$ is M\'{e}tivier, but it is not $H$-type. 
\hfill $\blacklozenge$
\end{example}

It should be pointed out that a pseudo-$H$-type Lie group is  a M\'{e}tivier group if and only if $\left\langle\cdot,\cdot\right\rangle_{\mathfrak{z}}$ is definite. 
In fact, if $\left\langle\cdot,\cdot\right\rangle_{\mathfrak{z}}$ is indefinite, there is a null vector $Z\in \mathfrak{z}\setminus \{0\}$ which satisfies $\left\langle Z, Z\right\rangle_{\mathfrak{z}}=0$ and hence $j(Z)=0$. 

\begin{example}\label{ex_pht}
We further introduce an example of the pseudo-$H$-type Lie groups. See e.g. \cite{autenried-furutani-markina-vasilev_2018} for other examples. 
We consider a Lie algebra $\mathfrak{g}\cong\RR^{6}$ generated by $X_{1}, X_{2},X_{3},X_{4},Z_{1},Z_{2}$ with commutation relations $\left[X_{1},X_{2}\right]=\left[X_{3},X_{4}\right]=Z_{1}$, $\left[X_{1},X_{3}\right]=\left[X_{2},X_{4}\right]=Z_{2}$, $\left[X_{1},X_{4}\right]=\left[X_{2},X_{3}\right]=\left[X_{i},Z_{j}\right]=0$, where $i=1,2,3,4$ and $j=1,2$. 
Then, $\mathfrak{z}=\mathrm{span}_{\mathbb{R}}\left\{Z_{1},Z_{2}\right\}$ is the centre of $\mathfrak{g}$. 
We consider the scalar product $\left\langle \cdot, \cdot\right\rangle$ on $\mathfrak{g}$ for which $X_{1}, X_{2},X_{3},X_{4},Z_{1},Z_{2}$ are orthogonal and 
\begin{equation*}
\langle X_{i},X_{j}\rangle_{\mathfrak{v}}
=
\begin{cases}
\,1 &\text{if $i=j$ and $i=1,2$},
\\
-1 &\text{if $i=j$ and $i=3,4$},
\\
\,0 &\text{if $i\neq j$},
\end{cases}
\hspace{20pt}
\langle Z_{i},Z_{j}\rangle_{\mathfrak{z}}
=
\begin{cases}
\,1 &\text{if $i=j=1$},
\\
-1 &\text{if $i=j=2$},
\\
\,0 &\text{if $i\neq j$}.
\end{cases}
\end{equation*}
Then, the orthogonal complement $\mathfrak{v}$ to $\mathfrak{z}$ is given as $\mathfrak{v}=\mathrm{span}_{\mathbb{R}}\left\{X_1, X_2, X_3, X_4\right\}$ and the restrictions $\langle\cdot,\cdot\rangle_{\mathfrak{v}}$ and $\langle\cdot,\cdot\rangle_{\mathfrak{z}}$ of the scalar product $\langle\cdot,\cdot\rangle$ to $\mathfrak{v}$ and $\mathfrak{z}$ are non-degenerate. 
We also see that, with respect to the basis $X_{1},X_{2},X_{3},X_{4}$ of $\mathfrak{v}$, the operators $j\left(Z_{1}\right)$, $j\left(Z_{2}\right)$ are written as
\begin{align*}
j\left(Z_{1}\right)
=
\begin{pmatrix}
0 & -1 & 0 & 0 \\
1 & 0 & 0 & 0 \\
0 & 0 & 0 & 1 \\
0 & 0 & -1 & 0 
\end{pmatrix},
\quad
j\left(Z_{2}\right)
=
\begin{pmatrix}
0 & 0 & 1 & 0 \\
0 & 0 & 0 & 1 \\
1 & 0 & 0 & 0 \\
0 & 1 & 0 & 0 
\end{pmatrix}.
\end{align*}
As is easily checked, the operators $j\left(Z_{1}\right)$, $j\left(Z_{2}\right)$ satisfy the conditions $\left(j\left(Z_{1}\right)\right)^{2}=-\langle Z_{1},Z_{1}\rangle_{\mathfrak{z}}\mathrm{id}_{\mathfrak{v}}$, $\left(j\left(Z_{2}\right)\right)^{2}=-\langle Z_{2},Z_{2}\rangle_{\mathfrak{z}}\mathrm{id}_{\mathfrak{v}}$, $j\left(Z_{1}\right)j\left(Z_{2}\right)=-j\left(Z_{2}\right)j\left(Z_{1}\right)$, and thus the Lie algebra $\mathfrak{g}$ is pseudo-$H$-type. 

\hfill $\blacklozenge$
\end{example}

\subsubsection{Heisenberg-Reiter groups}\label{subsec_HR}
We consider the so-called Heisenberg-Reiter Lie groups, cf. \cite{reiter_1974,torres_lopera_1985,torres_lopera_1988,butler_2003}. 

Consider a step-two nilpotent Lie group whose Lie algebra $\mathfrak{g}$ is given as the vector space direct sum of three subspaces $\mathfrak{u}_1$, $\mathfrak{u}_2$, $\mathfrak{w}$: $\mathfrak{g}=\mathfrak{u}_1\dot{+}\mathfrak{u}_2\dot{+}\mathfrak{w}$. 
We assume that there is a bilinear mapping $B: \mathfrak{u}_1\times \mathfrak{u}_2\rightarrow \mathfrak{w}$, which induces a Lie bracket on $\mathfrak{g}$ through 
\[
\left[\left(U_1, U_2, W\right), \left(U_1^{\prime}, U_2^{\prime}, W^{\prime}\right)\right]
=
\left(0, 0, B\left(U_1, U_2^{\prime}\right)-B\left(U_2,U_1^{\prime}\right)\right), 
\]
where $\left(U_1, U_2, W\right), \left(U_1^{\prime}, U_2^{\prime}, W^{\prime}\right)\in \mathfrak{u}_1\dot{+}\mathfrak{u}_2\dot{+}\mathfrak{w}$. 
We set $\mathfrak{u}_1^{\circ}:=\left\{U_1\in \mathfrak{u}_1\mid \forall U_2\in \mathfrak{u}_2,\; B\left(U_1, U_2\right)=0\right\}$ and $\mathfrak{u}_2^{\circ}:=\left\{U_2\in \mathfrak{u}_2\mid \forall U_1\in \mathfrak{u}_1,\; B\left(U_1, U_2\right)=0\right\}$. 
Then, the centre $\mathfrak{z}$ of $\mathfrak{g}$ is written as $\mathfrak{z}=\mathfrak{u}_1^{\circ}\dot{+}\mathfrak{u}_2^{\circ}\dot{+}\mathfrak{w}$.

Now, we consider a scalar product $\left\langle\cdot, \cdot \right\rangle$ on the Lie algebra $\mathfrak{g}$ and  assume that its restrictions $\left\langle \cdot, \cdot \right\rangle_{\mathfrak{u}_i}:=\left\langle \cdot, \cdot \right\rangle|_{\mathfrak{u}_i\times\mathfrak{u}_i}$ and $\left\langle \cdot, \cdot \right\rangle_{\mathfrak{w}}:=\left\langle \cdot, \cdot \right\rangle|_{\mathfrak{w}\times\mathfrak{w}}$ respectively to $\mathfrak{u}_i$, $i=1,2$, and $\mathfrak{w}$ are non-degenerate. 
We also assume that the subspaces $\mathfrak{u}_1$, $\mathfrak{u}_2$, $\mathfrak{w}$ are pairwise orthogonal to each other with respect to the scalar product $\left\langle\cdot, \cdot \right\rangle$. 
Then, the complementary subspace $\mathfrak{v}:=\left\{X\in \mathfrak{g}\mid \forall Z\in \mathfrak{z}\, \left\langle X, Z\right\rangle=0\right\}$ to $\mathfrak{z}$ is contained in $\mathfrak{u}_1\dot{+}\mathfrak{u}_2$ and we set $\mathfrak{v}_i:=\mathfrak{v}\cap\mathfrak{u}_i$, $i=1,2$. 
Clearly, we have $\mathfrak{v}=\mathfrak{v}_1\dot{+}\mathfrak{v}_2$, $\mathfrak{u}_i=\mathfrak{v}_i\dot{+}\mathfrak{u}_i^{\circ}$, $i=1,2$. 
The restriction $\left\langle \cdot, \cdot \right\rangle_{\mathfrak{v}}:=\left\langle \cdot, \cdot \right\rangle|_{\mathfrak{v}\times\mathfrak{v}}$ is a non-degenerate scalar product, but we also assume that the subspaces $\mathfrak{v}_i$ have the restricted non-degenerate scalar product $\left\langle \cdot, \cdot \right\rangle_{\mathfrak{v}_i}:=\left\langle \cdot, \cdot \right\rangle|_{\mathfrak{v}_i\times\mathfrak{v}_i}$, $i=1,2$. 
Note that, in this case, the restriction $\left\langle\cdot, \cdot \right\rangle_{\mathfrak{z}}$ of the scalar product $\left\langle\cdot, \cdot \right\rangle$ to $\mathfrak{z}$ is also non-degenerate. 

As in the previous section, we consider the {\it $j$-mapping} $j(Z): \mathfrak{v}\rightarrow \mathfrak{v}$ for all $Z\in\mathfrak{z}$ defined through $\left\langle j(Z)V, V^{\prime}\right\rangle_{\mathfrak{v}}=\left\langle Z, \left[V, V^{\prime}\right]\right\rangle_{\mathfrak{z}}$ where $V, V^{\prime}\in\mathfrak{v}$. 
For $W\in \mathfrak{w}$, we define $j_B^{12}\left(W\right): \mathfrak{v}_2\rightarrow \mathfrak{v}_1$ and $j_B^{21}\left(W\right): \mathfrak{v}_1\rightarrow \mathfrak{v}_2$ through 
\begin{align*}
\left\langle j_B^{12}\left(W\right)V_2, V_1\right\rangle_{\mathfrak{v}_1}=\left\langle j_B^{21}\left(W\right)V_1, V_2\right\rangle_{\mathfrak{v}_2}=\left\langle W, B\left(V_1, V_2\right)\right\rangle_{\mathfrak{w}},  
\end{align*}
where $V_1\in \mathfrak{v}_1$, $V_2\in \mathfrak{v}_2$. 
Then, the mapping $j(Z): \mathfrak{v}\rightarrow \mathfrak{v}$ can be given in terms of the direct sum decomposition $\mathfrak{v}=\mathfrak{v}_1\dot{+}\mathfrak{v}_2$ as 
\begin{align}
  \label{j-mapping for Heisenberg-Reiter}
\begin{pmatrix}
0 & -j_B^{12}(Z_0) \\
j_B^{21}(Z_0) & 0
\end{pmatrix},
\end{align}
where $Z\in \mathfrak{z}=\mathfrak{u}_1^{\circ}\dot{+}\mathfrak{u}_2^{\circ}\dot{+}\mathfrak{w}$ and $Z=\overline{Z}+Z_0$, with $\overline{Z}\in \mathfrak{u}_1^{\circ}\dot{+}\mathfrak{u}_2^{\circ}$, $Z_0\in\mathfrak{w}$. 

\subsubsection{Step-two nilpotent Lie groups associated to semi-simple modules}\label{Step-two Lie groups associated to semi-simple modules}
Here, we think of the step-two nilpotent Lie groups associated to representations of semi-simple Lie groups as considered by P. Eberlein \cite{eberlein_2008}. 
Let $\mathfrak{z}$ be a semi-simple Lie algebra and $j:\mathfrak{z}\rightarrow \mathrm{End}\left(\mathfrak{v}\right)$ a real representation of the Lie algebra $\mathfrak{z}$ on a vector space $\mathfrak{v}$. 
We suppose that the direct sum $\mathfrak{g}=\mathfrak{v}\dot{+}\mathfrak{z}$ is equipped with an inner product $\left\langle \cdot ,\cdot\right\rangle: \mathfrak{g}\times \mathfrak{g}\rightarrow \mathbb{R}$ for which the followings are satisfied:

\begin{enumerate}[{(}E$\,1${)}]
\item\label{cond_e1} The vector subspaces $\mathfrak{v}$ and $\mathfrak{z}$ are orthogonal to each other with respect to the inner product $\left\langle \cdot ,\cdot\right\rangle$. 
\item\label{cond_e2} For all $Z\in \mathfrak{z}$, we have $\left\langle j(Z)V, V^{\prime}\right\rangle+\left\langle V, j(Z)V^{\prime}\right\rangle=0$, where $V, V^{\prime}\in \mathfrak{v}$. 
\item\label{cond_e3} For all $Z_1, Z_2, Z_3\in\mathfrak{z}$, we have $\left\langle \mathrm{ad}_{Z_1}Z_2, Z_3\right\rangle+\left\langle Z_2, \mathrm{ad}_{Z_1}Z_3\right\rangle=0$, where $\mathrm{ad}_{Z_1}:\mathfrak{z}\rightarrow \mathfrak{z}$ denotes the adjoint action on $\mathfrak{z}$. 
\end{enumerate}

Now, we consider the Lie bracket $\left[\cdot, \cdot\right]: \mathfrak{g}\times \mathfrak{g}\rightarrow \mathfrak{g}$ defined through 
\[
\left\langle \left[V+Z, V^{\prime}+Z^{\prime}\right], Z^{\prime\prime}\right\rangle
=
\left\langle j\left(Z^{\prime\prime}\right), \left[V, V^{\prime}\right]_{\mathfrak{z}}\right\rangle, 
\]
where $V, V^{\prime}\in\mathfrak{v}$, $Z, Z^{\prime}\in \mathfrak{z}$. 
Here, $\left[\cdot, \cdot\right]_{\mathfrak{z}}:\mathfrak{z}\times \mathfrak{z}\rightarrow \mathfrak{z}$ is the Lie bracket on $\mathfrak{z}$. 
Following \cite{eberlein_2008}, the Lie algebra $\left(\mathfrak{g}, \left[\cdot, \cdot\right]\right)$ is called a \textit{real step-two nilpotent Lie algebra arising from the semi-simple module} $\left(\mathfrak{v}, j\right)$. 
Note that $\mathfrak{z}$ is not a Lie subalgebra of $\mathfrak{g}$ with respect to the Lie bracket $\left[\cdot, \cdot\right]: \mathfrak{g}\times \mathfrak{g}\rightarrow \mathfrak{g}$. 
Clearly, the $j$-mapping for the Lie algebra $\left(\mathfrak{g}, \left[\cdot, \cdot\right]\right)$ is nothing but the representation $j: \mathfrak{z}\rightarrow \mathrm{End}\left(\mathfrak{v}\right)$. 

\section{Equilibrium points and stability}\label{Equilibrium points and stability}
\subsection{Williamson types and stability of equilibrium points} 
In this subsection, we introduce the Williamson type of an equilibrium point for a Hamiltonian system $\left( N,\omega,H \right)$ on a $2n$-dimensional symplectic manifold. 
To analyse the stability of equilibrium points for the equation \eqref{eq_geodesic_flow} in the subsequent subsections, we further describe the relation between the Williamson types of equilibrium points and their Lyapunov stability. 
See e.g. \cite{williamson_1936}, \cite{ratiu_tarama_2015}, \cite{ratiu_tarama_2020}, \cite[Appendix $6$]{arnold_1989}, \cite{zung_1996} for the related facts on these materials.

\paragraph{Williamson types}
We recall the notion of Williamson type of an equilibrium point for a linear Hamiltonian system on a $2n$-dimensional symplectic vector space $\displaystyle \left( \RR^{2n}, \sum_{i=1}^n\dd q_i\wedge\dd p_i \right)$. 
For the Hamiltonian $H\coloneq\frac{1}{2}z^{\mathrm{T}}Sz$, $z=(q_1,\dots,q_n,p_1,\dots,p_n)^{\mathrm{T}}\in \RR^{2n}$, where $S$ is a $2n\times 2n$ constant symmetric matrix, the Hamilton equation is given by the linear differential equation 
\begin{align}
  \frac{\dd z}{\dd t}=\mathsf{J}^{\mathrm{T}}S z,\label{lin_Ham_eq}
\end{align}
where $z\in\RR^{2n}$ and $\mathsf{J}=\begin{pmatrix} 0 & -\mathsf{E}_n \\ \mathsf{E}_n & 0 \end{pmatrix}$ is the symplectic matrix with $\mathsf{E}_n$ being the $n\times n$ unit matrix. 

For the linear Hamiltonian system \eqref{lin_Ham_eq}, the origin $\bm{0}\in \mathbb{R}^{2n}$ is an equilibrium point. 
The eigenvalues of the Hamiltonian matrix $A:=\mathsf{J}^{\mathrm{T}}S$ in the equation \eqref{lin_Ham_eq}  are among the following four kinds of complex numbers: 
\begin{align*}
  \arraycolsep=15pt
  \begin{array}{rrrr}
    \left( 1 \right)\; \pm\mu_i\sqrt{-1}, & \left( 2 \right)\; \pm\mu_r, & \left( 3 \right)\; \pm\mu_r\pm\mu_i\sqrt{-1}, & \left( 4 \right)\; 0,
  \end{array}
\end{align*}
where $\mu_r,\mu_i\in\RR\setminus\{0\}$. 
Note that the types $\left( 1 \right)$, $\left( 2 \right)$, and $\left( 4 \right)$ are given as pairs, whereas the type $\left( 3 \right)$ is a quadruple. 
The three types $\left( 1 \right)$, $\left( 2 \right)$, and $ \left( 3 \right)$ of the eigenvalues correspond to the \textit{elliptic}, \textit{hyperbolic}, and \textit{focus-focus} behaviour of the dynamics around the equilibrium point $\bm{0}$, respectively. 
We denote the number of the four types $(1), (2), (3), (4)$ of eigenvalues by $2k_e,\,2k_h,\,4k_f,\,2k_0$, respectively. 
We clearly have $k_e+k_h+2k_f+k_0=n$. 
If the eigenvalues of $A$ are all non-zero, the equilibrium point $\bm{0}$ is isolated and the triple $\left( k_e,k_h,k_f \right)$ is called the \textit{Williamson type} of the equilibrium point $\bm{0}$. 
In this case, we have $k_e+k_h+2k_f=n$. 
Below, we only focus on those isolated equilibrium points.  

\medskip

In general, the Williamson type of an equilibrium point is defined  for a Hamiltonian system $\left( N,\omega,H \right)$ on a $2n$-dimensional symplectic manifold as follows: \\
Let $x\in N$ be an equilibrium point of this Hamiltonian system and we linearize the Hamilton equation around $x$ as 
\begin{align}\label{linearized_eq}
  \frac{\dd P}{\dd t}=\omega\left( x \right)^{-1}\mathrm{Hess}_{x}\left( H \right)P,
\end{align}
where $P\in T_x N$. Here, $\omega(x)$ is the symplectic form at $x$ and $\mathrm{Hess}_x\left( H \right)$ is the Hessian of the Hamiltonian $H$ at $x$, both of which are regarded as the linear operators of $T_x N$ to $T^{\ast}_x N$. 
Note that the differential equation \eqref{linearized_eq} is also a linear Hamilton equation. 
See \cite[Subsection $4.1$]{ratiu_tarama_2020} for more details. 
With respect to a Darboux coordinate system of $T_x N$, the symplectic form $\omega\left( x \right)$ is written as the matrix $\mathsf{J}=\begin{pmatrix} 0 &-\mathsf{E}_n \\ \mathsf{E}_n & 0 \end{pmatrix}$. 
Then, the matrix representation of the linear operator $\omega\left( x \right)^{-1}\mathrm{Hess}_{x}\left( H \right)$ is a Hamiltonian matrix, as considered in the case of the linear Hamilton equation \eqref{lin_Ham_eq}. 
The Williamson type of the isolated equilibrium point $x$ of the Hamiltonian system $\left( N,\omega,H \right)$ is defined as the one of the linearized Hamilton equation \eqref{linearized_eq} or equivalently \eqref{lin_Ham_eq} with $A=\mathsf{J}^{-1}S=\omega\left( x \right)^{-1}\mathrm{Hess}_{x}\left( H \right)$. 

\paragraph{Relation to instability}
As is well-known in the theory of ordinary differential equations, an equilibrium point for a system of differential equations is asymptotically stable if all the eigenvalues of the linearization matrix at the equilibrium point have negative real parts. 
On the other hand, it is unstable if at least one of the eigenvalues of the linearization matrix has a positive real part. 
See e.g. \cite[Chapter $5$]{pontryagin_1962} and \cite[Chapter $9$]{hirsch_smale_1974} for the details about these basic theorems. 
Concerning a Hamiltonian system, an equilibrium point whose Williamson type $(k_e, k_h, k_f)$ with $k_h\neq 0$ or $k_f\neq 0$ is obviously unstable, since at least one of the eigenvalues of the linearization matrix has a positive real part. 
\begin{proposition}
Let $x\in N$ be an isolated equilibrium point of a Hamiltonian system $\left( N,\omega, H \right)$ with the Williamson type $\left( k_e,k_h,k_f \right)$.
If either $k_h$ or $k_f$ is non-zero, the equilibrium point $x$ is unstable.
\hfill $\square$
\end{proposition}

\paragraph{Relation to Lyapunov stability of linear systems}
As for linear Hamiltonian system, if the coefficient matrix $\mathsf{J}^{\mathrm{T}}S$ of the Hamilton equation \eqref{lin_Ham_eq} has only purely imaginary eigenvalues (particularly if $k_h=k_f=0$ when the equilibrium point has the Williamson type $(k_e,k_h,k_f)$) and if the Hamiltonian matrix $\mathsf{J}^{\mathrm{T}}S$ is simple (diagonalizable), the associated Jordan normal form does not have nilpotent parts. 
After a change of coordinates, the solution $z\left( t \right)$ to the Hamilton equation \eqref{lin_Ham_eq} is explicitly written with a block diagonal matrix as
\begin{align*}
  z\left( t \right)=
  \begin{pmatrix}
    R\left( \mu_1 t \right) & & & \\
    & \ddots & & \\
    & & R\left( \mu_{k_e} t \right) & \\
    & & & \mathsf{E}_{2(n-k_e)}
  \end{pmatrix}
  z\left(0\right),
\end{align*}
where $R\left( \theta \right)$ stands for the $2$-dimensional rotation matrix: $R\left( \theta \right)=\begin{pmatrix*}[r] \cos\theta & \sin\theta \\ -\sin\theta & \cos\theta \end{pmatrix*}$, and hence the equilibrium point is Lyapunov stable. 
\begin{proposition}
Suppose that the Williamson type of the equilibrium point $z=\bm{0}$ for the linear Hamiltonian system \eqref{lin_Ham_eq} is given as $\left( k_e,k_h,k_f \right)$. 
The equilibrium point $z=\bm{0}$ is Lyapunov stable, if $k_h=k_f=0$ and the linearization matrix at $z=\bm{0}$ is simple.
\hfill $\square$
\end{proposition}

\begin{remark}
In the general framework of dynamical systems theory, an equilibrium point for a system is usually called \textit{elliptic} (respectively \textit{hyperbolic}), if the linearization matrix is diagonalizable and has only purely imaginary eigenvalues (respectively if it has only eigenvalues whose real parts are non-zero). 
See e.g. \cite[Section $6$]{meyer_offin_2017}. 
Concerning Hamiltonian systems, it is common to classify the eigenvalues of the linearization matrix into purely imaginary, real, and complex eigenvalues. 
Because of the finer classification, the terminology ``hyperbolic'' for Hamiltonian systems is unfortunately inconsistent with the one in general dynamical systems. 
The equilibrium point of a Hamiltonian system whose Williamson type $(k_e, k_h, k_f)$ satisfies $k_e=0$ (or equivalently, $k_{h}\neq 0$ or $k_{f}\neq 0$) is hyperbolic in the general framework of dynamical systems theory. 
On the other hand, an equilibrium point of a Hamiltonian system, which is ``hyperbolic'' in view of general dynamical systems theory, is not necessarily those equilibrium points with the Williamson type $(k_e, k_h, k_f)=(0,k_h, 0)$. 
To avoid possible confusions, we do not use the terminology ``hyperbolic equilibrium'' in the present paper. 
Nevertheless, if the linearization matrix is diagonalizable, the equilibrium point of a Hamiltonian system whose Williamson type $(k_e, k_h, k_f)$ satisfies $k_h=k_f=0$ is elliptic in the terminology of general dynamical systems theory and visa versa. 
\end{remark}

\subsection{Classification by the conjugacy classes of Cartan subalgebras in $\mathfrak{so}(p,q)$}\label{subsec_cartan}

In this subsection, we review the classification about the conjugacy classes of Cartan subalgebras in the Lie algebra $\mathfrak{so}(p,q)$, which is used in the analysis of Williamson types for the relative equilibria of the geodesic flows of step-two nilpotent Lie groups. 
Below, we assume the condition $p\leq q$ without loss of generality. 

By the classification of Cartan subalgebras in \cite{sugiura_1959}, which is summarized in Appendix \ref{appendix_a}, the Cartan subalgebras in $\mathfrak{so}(p,q)$ are, with $p+q$ being even, conjugate to one of the sets consisting of the following matrices: 
\begin{equation}\label{Cartan_D_n_standard}
\begin{pmatrix}
D_1+D_2 & 0 & 0 & 0 & 0 & 0 & 0 \\
0 & D_4 & 0 & 0 & -D_3 & 0 & 0 \\
0 & 0 & 0 & 0 & 0 & -D_5 & 0 \\
0 & 0 & 0 & D_1-D_2 & 0 & 0 & 0 \\
0 & -D_3 & 0 & 0 & D_4 & 0 & 0 \\
0 & 0 & -D_5 & 0 & 0 & 0 & 0 \\
0 & 0 & 0 & 0 & 0 & 0 & D_6
\end{pmatrix},
\end{equation}
where 
\begin{align*}
D_1
&=
\begin{pmatrix}
0 & \mathrm{diag}\left(-h_1, \ldots, -h_{\ell}\right) \\
\mathrm{diag}\left(h_1, \ldots, h_{\ell}\right) & 0
\end{pmatrix}, \notag \\
D_2
&=
\begin{pmatrix}
0 & \mathrm{diag}\left(-h_{\ell +1}, \ldots, -h_{2\ell}\right) \\
\mathrm{diag}\left(h_{\ell +1}, \ldots, h_{2\ell}\right) & 0
\end{pmatrix}, \notag \\
D_3
&=
\mathrm{diag}\left(h_{2\ell +1}, \ldots, h_{2\ell +k}, h_{2\ell +1}, \ldots, h_{2\ell +k}\right), \notag \\
D_4
&=
\begin{pmatrix}
0 & \mathrm{diag}\left(-h_{2\ell +k+1}, \ldots, -h_{2\ell+2k}\right) \\
\mathrm{diag}\left(h_{2\ell +k+1}, \ldots, h_{2\ell+2k}\right) & 0
\end{pmatrix}, \notag \\
D_5
&=
\mathrm{diag}\left(h_{2\ell+2k+1}, \ldots, h_p\right), \notag \\
D_6
&=
\begin{pmatrix}
0 & \mathrm{diag}\left(-h_{p +1}, \ldots, -h_{n}\right) \\
\mathrm{diag}\left(h_{p +1}, \ldots, h_{n}\right) & 0
\end{pmatrix}, 
\end{align*}
and  $0\leq k, 0\leq \ell, 2\left(k+\ell\right)\leq p$. Here, we put $p+q=2n$. 
Note that the matrices \eqref{Cartan_D_n_standard} correspond to \eqref{Cartan_D_n} in Appendix \ref{appendix_a}.
There is another conjugacy class, which corresponds to \eqref{Cartan_D_n_bis} in Appendix \ref{appendix_a}, represented by the set of all the following matrices: 
\begin{equation}\label{Cartan_D_n_bis_standard}
\begin{pmatrix}
D_4 & 0 & -D_3 & 0 \\
0 & D_5^{\prime}+D_7 & 0 & 0 \\
-D_3 & 0 & D_4 & 0 \\
0 & 0 & 0 & D_5^{\prime}-D_7
\end{pmatrix},
\end{equation}
where 
\begin{align*}
D_3
&=
\mathrm{diag}\left(h_{1}, \ldots, h_{(n-2)/2}, h_{1}, \ldots, h_{(n-2)/2}\right), \notag \\
D_4
&=
\begin{pmatrix}
0 & \mathrm{diag}\left(-h_{n/2}, \ldots, -h_{n-2}\right) \\
\mathrm{diag}\left(h_{n/2}, \ldots, h_{n-2}\right) & 0
\end{pmatrix}, \notag \\
D_5^{\prime}
&=
\mathrm{diag}\left(h_{n-1}, -h_{n-1}\right), \notag \\
D_7
&=
\begin{pmatrix}
0 & -h_n \\
h_n & 0
\end{pmatrix}. 
\end{align*}

When $p+q$ is odd, the Cartan subalgebras in $\mathfrak{so}(p,q)$ are conjugate to one of the sets consisting of \eqref{Cartan_D_n_standard} with zero vectors being added to the last row and the last column, or those consisting of the following matrices: 

\begin{equation}\label{Cartan_B_n_standard}
\begin{pmatrix}
D_1+D_2 & 0 & 0 & 0 & 0 & 0 & 0 & 0 \\
0 & D_4 & 0 & 0 & -D_3 & 0 & 0 & 0 \\
0 & 0 & 0 & 0 & 0 & -D_5^{\prime\prime} &0 & 0 \\
0 & 0 & 0 & D_1-D_2 & 0 & 0 & 0 & 0 \\
0 & -D_3 & 0 & 0 & D_4  & 0 & 0 & 0 \\
0 & 0 & -D_5^{\prime\prime} & 0 & 0 & 0 & 0 & D_8/\sqrt{2} \\
0 & 0 & 0 & 0 & 0 & 0 & D_6 & 0 \\
0 & 0 & 0 & 0 & 0 & -D_8^{\mathrm{T}}/\sqrt{2} & 0 & 0
\end{pmatrix},
\end{equation}
where $D_1, D_2, D_3, D_4, D_6$ are the same matrices as above and $D_5^{\prime\prime}=\mathrm{diag}\left(0, h_{2\left(\ell +k+1\right)}, \ldots, h_p\right)$, and $D_8=\left(h_{2\ell +2k+1}, 0, \ldots, 0\right)^{\mathrm{T}}\in \mathbb{R}^{p-2\ell-2k}$, and $0\leq k, 0\leq \ell, 2\left(k+\ell\right)+1\leq p$. Here, we put $p+q=2n+1$. Note also that the matrices \eqref{Cartan_B_n_standard} correspond to \eqref{Cartan_B_n} in Appendix \ref{appendix_a}.

\subsection{Williamson types and stability of equilibrium points for the geodesic flows of step-two nilpotent Lie groups.}\label{subsec_wt_and_stb}
In this subsection, we investigate the Williamson types of equilibrium points for the geodesic flows and their stability on the step-two nilpotent Lie groups. 

As in \S \ref{Lie-Poisson equation for step-two nilpotent Lie groups}, the Lie-Poisson equation for the geodesic flow of a step-two nilpotent Lie group with respect to the left-invariant pseudo-Riemannian metric is given as \eqref{eq_geodesic_flow}.
It is easily checked that the equilibrium points of the equation \eqref{eq_geodesic_flow} are the points $Y_{\mathfrak{v}}+Y_{\mathfrak{z}}\in\mathfrak{v}\dot{+}\mathfrak{z}$, where $Y_{\mathfrak{v}}\in \ker\left(j\left(Y_{\mathfrak{z}}\right)\right)$. 
The Lie-Poisson equation \eqref{eq_geodesic_flow} can be restricted to coadjoint orbits, as is explained in \S \ref{subsec_lp} prior to Remark \ref{remark on degenerate case}. 
On a coadjoint orbit $\mathcal{O}$, the centre component $Y_{\mathfrak{z}}$ of a point $Y\in \mathcal{O}\subset \mathfrak{g}$ is constant, as the functions $f(Y_{\mathfrak{z}})$ which only depend on the centre component are Casimir functions with respect to the Lie-Poisson bracket. 
The Lie-Poisson equation \eqref{eq_geodesic_flow} can be restricted to $\mathfrak{v}+Y_{\mathfrak{z}}$, where the restricted system is again described by the first component of \eqref{eq_geodesic_flow} under the identification through the translation by $Y_{\mathfrak{z}}$. 
Moreover, since the coadjoint orbits are linear submanifolds of $\mathfrak{v}$ from Proposition \ref{coadjoint orbit of step-two}, the restriction of the first component of \eqref{eq_geodesic_flow} to each coadjoint orbit is given as
\begin{align}\label{lie-poisson on orbits}
\frac{\dd Y}{\dd t}=\left(j\left( Y_{\mathfrak{z}} \right)|_{\mathcal{O}}\right)Y, \quad Y\in \mathcal{O}. 
\end{align}
In particular, on each generic, i.e. maximal-dimensional, coadjoint orbit, the restricted equation \eqref{lie-poisson on orbits} has only one equilibrium point, which is 
$Y_{\mathfrak{v}}+Y_{\mathfrak{z}}$, where $Y_{\mathfrak{v}}\in \ker\left(j\left(Y_{\mathfrak{z}}\right)\right)$. We can further see that $j\left( Y_{\mathfrak{z}} \right)$ and $j\left( Y_{\mathfrak{z}} \right)|_{\mathcal{O}}$ have the same eigenvalues. 
Thus, the Williamson type of the equilibrium point on each generic coadjoint orbit is determined through the eigenvalues of $j\left( Y_{\mathfrak{z}} \right)$, as we consider in what follows. 

\subsubsection{Williamson types for general step-two nilpotent Lie groups}\label{subsec_wt_gen}
As mentioned in Introduction, there are only less unified approaches to nilpotent Lie groups, compared to semi-simple Lie groups. 
However, for the step-two nilpotent Lie groups, it is possible to determine the Williamson types of the equilibrium points for the Lie-Poisson equation on each generic coadjoint orbit even in the general case.
In view of the results in \S \ref{subsec_lp} and \S \ref{subsec_cartan}, we can generally classify the Williamson types of the equilibrium points for the Lie-Poisson equation, by determining the types of eigenvalues of the matrices \eqref{Cartan_D_n_standard}, \eqref{Cartan_D_n_bis_standard}, \eqref{Cartan_B_n_standard} in a suitable Cartan subalgebra.
Consequently, we have the following theorem. 

\begin{theorem}\label{main_thm_general}
The equilibrium $Y_{\mathfrak{v}}+Y_{\mathfrak{z}}\in \mathcal{O}$, where $Y_{\mathfrak{v}}\in \ker\left(j\left(Y_{\mathfrak{z}}\right)\right)$, for the Lie-Poisson equation \eqref{eq_geodesic_flow}, restricted to the orbit $\mathcal{O}$, has the Williamson type $\left(2\ell-p+n,p-2\ell-2k,k\right)$, $\left(0, 0, n/2\right)$, and $\left(2\ell+1+n-p, p-2\ell-2k-1,k\right)$ if $j\left(Y_{\mathfrak{z}}\right)$ is conjugate to \eqref{Cartan_D_n_standard}, \eqref{Cartan_D_n_bis_standard}, \eqref{Cartan_B_n_standard}, respectively, and if the rank of the matrices \eqref{Cartan_D_n_standard}, \eqref{Cartan_D_n_bis_standard}, \eqref{Cartan_B_n_standard} is equal to $2n$, respectively. 
\hfill $\square$
\end{theorem}
\begin{remark}
In the case where the rank of the matrices \eqref{Cartan_D_n_standard}, \eqref{Cartan_D_n_bis_standard}, \eqref{Cartan_B_n_standard} is less than $2n$, the Williamson type in Theorem \ref{main_thm_general} appropriately changes.
\hfill $\square$
\end{remark}

\subsubsection{Williamson types for some concrete classes of step-two nilpotent Lie groups}\label{subsec_wt_conc}

We here determine the Williamson types of the equilibrium points $Y_{\mathfrak{v}}+Y_{\mathfrak{z}}\in\mathfrak{v}\dot{+}\mathfrak{z}$, where $Y_{\mathfrak{v}}\in\ker\left(j\left(Y_{\mathfrak{z}}\right)\right)$, for the Lie-Poisson equation of the concrete classes of step-two nilpotent Lie groups discussed in \S \ref{Certain classes of step-two nilpotent Lie groups}. 
The method is based on the eigenvalues of the operator $j\left( Y_{\mathfrak{z}} \right)\in\mathfrak{so}\left( \mathfrak{v},\langle\cdot,\cdot\rangle_{\mathfrak{v}} \right)$.

\paragraph{Step-two Carnot groups}
As for the step-two Carnot group, the operator $j\left(Y_{\mathfrak{z}}\right)$ is not full rank only in the case where $Y_{\mathfrak{z}}=0$. In this case, it is easily checked that $j\left(Y_{\mathfrak{z}}\right)=0$. We assume that $Y_{\mathfrak{z}}\neq0$. Then, the equilibrium points are given as $0+Y_{\mathfrak{z}}\in\mathfrak{v}\dot{+}\mathfrak{z}$ since $j\left(Y_{\mathfrak{z}}\right)$ is full rank operator for $Y_{\mathfrak{z}}\in\mathfrak{z}\setminus\{0\}$. For the same reason as before, $j\left(Y_{\mathfrak{z}}\right)$ is conjugate to \eqref{Cartan_D_n_standard} or \eqref{Cartan_D_n_bis_standard}. 
The dimension of $\mathfrak{v}$ is even by the argument in \S \ref{subsubsubsection_2_3_1} and, using Theorem \ref{main_thm_general}, we can determine the Williamson type of the equilibria as follows: 

\begin{corollary}\label{cor_carnot}
The equilibrium point $0+Y_{\mathfrak{z}}$ for the Lie-Poisson equation has the Williamson type $\left(2\ell-p+n,p-2\ell-2k,k\right)$ and $\left(0, 0, n/2\right)$, if $j\left(Y_{\mathfrak{z}}\right)$ is conjugate to \eqref{Cartan_D_n_standard} and \eqref{Cartan_D_n_bis_standard}, respectively. 
In particular, the equilibrium points for M\'{e}tivier groups are elliptic and hence Lyapunov stable. 

\hfill $\square$
\end{corollary}

\paragraph{$H$-type and pseudo-$H$-type Lie groups}
Let $G$ be a pseudo-$H$-type Lie group equipped with a left-invariant pseudo-Riemannian metric, i.e. $G$ is a step-two nilpotent Lie group with the condition \eqref{relation_pseudo-h-type}.

By using the relation \eqref{relation_pseudo-h-type} and the skew-symmetry of $j\left( Z \right)$ for all $Z\in\mathfrak{z}$ with respect to the induced scalar product $\langle\cdot,\cdot\rangle_{\mathfrak{v}}$ on $\mathfrak{v}$, we have
\begin{equation}
\left(j\left( Z \right)\right)^2=-\langle Z,Z \rangle_{\mathfrak{z}}\mathrm{id}_{\mathfrak{v}}, \qquad Z\in\mathfrak{z}.
\end{equation}
This relation immediately implies that all the eigenvalues of $j\left( Z \right)$ where $Z\in\mathfrak{z}$ are purely imaginary (respectively real) if $\langle Z,Z \rangle_{\mathfrak{z}}$ is positive (respectively negative).

On a generic coadjoint orbit\footnote{By generic coadjoint orbits, we mean those passing through $Y=Y_{\mathfrak{v}}+Y_{\mathfrak{z}}$ satisfying $\langle Y_{\mathfrak{z}},Y_{\mathfrak{z}} \rangle_{\mathfrak{z}}\neq0$}, the equilibrium point is given as $0+Y_{\mathfrak{z}}$ with $\langle Y_{\mathfrak{z}},Y_{\mathfrak{z}} \rangle_{\mathfrak{z}}\neq0$, and thus the stability is charactarized as follows:
\begin{theorem}\label{stability_pseudo-h-type}
  For the pseudo-$H$-type Lie groups, the stability of the equilibrium points $0+Y_{\mathfrak{z}}\in\mathfrak{v}\dot{+}\mathfrak{z}$ for the Lie-Poisson equation \eqref{eq_geodesic_flow} are described as follows:
  \begin{enumerate}
    \item If $\langle Y_{\mathfrak{z}},Y_{\mathfrak{z}} \rangle_{\mathfrak{z}}<0$, $j\left( Y_{\mathfrak{z}} \right)$ has only real eigenvalues, and thus the equilibrium point $0+Y_{\mathfrak{z}}\in\mathfrak{v}\dot{+}\mathfrak{z}$ is unstable.
    \item If $\langle Y_{\mathfrak{z}},Y_{\mathfrak{z}} \rangle_{\mathfrak{z}}>0$, $j\left( Y_{\mathfrak{z}} \right)$ has only purely imaginary eigenvalues, and thus the equilibrium point $0+Y_{\mathfrak{z}}\in\mathfrak{v}\dot{+}\mathfrak{z}$ is Lyapunov stable.
  \end{enumerate}
  \hfill $\square$
\end{theorem}

We verify Theorem \ref{stability_pseudo-h-type} for the example of pseudo-$H$-type Lie groups, which was considered at the end of \S \ref{subsubsubsection_2_3_1}. Recall that, in Example \ref{ex_pht}, \S \ref{subsubsubsection_2_3_1}, we consider a pseudo-$H$-type Lie algebra $\mathfrak{g}\cong\RR^{6}$ generated by $X_{1}, X_{2},X_{3},X_{4},Z_{1},Z_{2}$ with commutation relations $\left[X_{1},X_{2}\right]=\left[X_{3},X_{4}\right]=Z_{1}$, $\left[X_{1},X_{3}\right]=\left[X_{2},X_{4}\right]=Z_{2}$, $\left[X_{1},X_{4}\right]=\left[X_{2},X_{3}\right]=\left[X_{i},Z_{j}\right]=0$, where $i=1,2,3,4$ and $j=1,2$ and a scalar product on it defined through $\langle X_{1},X_{1}\rangle=\langle X_{2},X_{2}\rangle=\langle Z_{1},Z_{1}\rangle=1$, $\langle X_{3},X_{3}\rangle=\langle X_{4},X_{4}\rangle=\langle Z_{2},Z_{2}\rangle=-1$, $\langle X_{i},X_{j}\rangle=\langle X_{i},Z_{k}\rangle=\langle Z_{1},Z_{2}\rangle=0$, where $i,j=1,2,3,4$ with $i\neq j$ and $k=1,2$.

If we write $Y_{\mathfrak{z}}=aZ_{1}+bZ_{2}\in\mathfrak{z}=\mathrm{span}_{\mathbb{R}}\left\{Z_{1},Z_{2}\right\}$ for the equilibrium point $0+Y_{\mathfrak{z}}\in\mathfrak{v}\dot{+}\mathfrak{z}$, $j\left(Y_{\mathfrak{z}}\right)$ is concretely written as
\begin{align*}
j\left(Y_{\mathfrak{z}}\right)
=
\begin{pmatrix}
0 & -a & b & 0 \\
a & 0 & 0 & b \\
b & 0 & 0 & a \\
0 & b & -a & 0 
\end{pmatrix}.
\end{align*}
In this case, the eigenvalues of $j\left(Y_{\mathfrak{z}}\right)$ are given as $\displaystyle\pm\sqrt{b^{2}-a^{2}}=\pm\sqrt{-\langle Y_{\mathfrak{z}},Y_{\mathfrak{z}}\rangle_{\mathfrak{z}}}$. Thus, if $\langle Y_{\mathfrak{z}},Y_{\mathfrak{z}}\rangle_{\mathfrak{z}}=a^{2}-b^{2}<0$ (respectively if $\langle Y_{\mathfrak{z}},Y_{\mathfrak{z}}\rangle_{\mathfrak{z}}=a^{2}-b^{2}>0$), $j\left(Y_{\mathfrak{z}}\right)$ has only real (repectively only purely imaginary) eigenvalues.

The following corollary is immediate from Theorem \ref{stability_pseudo-h-type}, since for the $H$-type Lie groups, the induced scalar product $\langle\cdot,\cdot\rangle_{\mathfrak{z}}$ on $\mathfrak{z}$ is positive-definite. 
\begin{corollary}\label{cor_ht}
For the $H$-type Lie groups, the equilibrium points $0+Y_{\mathfrak{z}}\in\mathfrak{v}\dot{+}\mathfrak{z}$, with $Y_{\mathfrak{z}}\neq 0$, of the Lie-Poisson equation \eqref{eq_geodesic_flow} is Lyapunov stable. 
\hfill $\square$
\end{corollary}

In particular, the $j$-mapping $J_{\zeta}$ for the Heisenberg Lie algebra $\mathfrak{h}_{2n+1}$ in \S \ref{subsec_heis} has the eigenvalue $\pm \sqrt{-1}\zeta$, where $\zeta\in \mathbb{R}$ is the coordinate in the centre $\mathfrak{z}$ of the Lie algebra, and hence the equilibrium point is Lyapunov stable as mentioned in Corollary \ref{cor_ht}. 

\paragraph{Heisenberg-Reiter groups}
We consider a Heisenberg-Reiter group $G$ whose Lie algebra is decomposed into the direct sum $\mathfrak{g}=\mathfrak{v}\dot{+}\mathfrak{z}$ of the centre $\mathfrak{z}$ and the complement $\mathfrak{v}$ as in \S \ref{subsec_HR}. 
The complement allows the orthogonal direct sum decomposition as $\mathfrak{v}=\mathfrak{v}_1\dot{+}\mathfrak{v}_2$ and we put $m_i\coloneq\dim\mathfrak{v}_i$, $i=1,2$, and $\left( p_1,q_1 \right)\coloneq \mathrm{sign}\left( \langle \cdot,\cdot\rangle_{\mathfrak{v}_1} \right)$, $\left( p_2,q_2 \right)\coloneq \mathrm{sign}\left( \langle \cdot,\cdot\rangle_{\mathfrak{v}_2} \right)$. 
Clearly, we have $m_i=p_i+q_i$, $i=1,2$. 
Without loss of generality, we assume $m_1 \geq m_2$. 

To find the types of the eigenvalues of the $j$-mapping $j(Z)=\begin{pmatrix}0 & -j_B^{12}(Z_0) \\ j_B^{21}(Z_0) & 0\end{pmatrix}$, as in \eqref{j-mapping for Heisenberg-Reiter}, we here consider an analogue of the singular value decomposition for the operator $\Phi\coloneq j_B^{12}(Z_0)$ through orthogonal transformations with respect to the indefinite scalar products $\langle\cdot,\cdot\rangle_{\mathfrak{v}_1}$ and $\langle\cdot,\cdot\rangle_{\mathfrak{v}_2}$. 
Such an extension of the singular value decompositions is discussed in \cite{hassi_1991}. 
The following proof is along the line of \cite[Theorem 2.4]{hassi_1991}, but we considered a more detailed version of it. 

As a technical assumption, we suppose that the operator $\Phi^{\ast}\circ \Phi\left( =j_B^{21}\left(Z_0\right)j_B^{12}\left(Z_0\right) \right)$ is diagonalizable over the real number field and that the image $\mathrm{Im}\left( \Phi \right)$ of $\Phi=j_B^{12}(Z_0)\colon\mathfrak{v}_2\to\mathfrak{v}_{1}$ is a non-degenerate linear subspace of $\left( \mathfrak{v}_1,\langle\cdot,\cdot\rangle_{\mathfrak{v}_1} \right)$. 
Since $\Phi^{\ast}\circ \Phi$ is diagonalizable over $\mathbb{R}$, we take a set of the eigenvectors $B_1, \dots, B_{m_2}$ for $\Phi^{\ast}\circ \Phi$ respectively belonging to the eigenvalues $\lambda_1, \dots, \lambda_{m_2}\in\mathbb{R}$ which forms an orthogonal basis of $\left( \mathfrak{v}_2,\langle\cdot,\cdot\rangle_{\mathfrak{v}_2} \right)$ satisfying
\begin{align*}
  \langle B_i,B_j\rangle_{\mathfrak{v}_2}=
  \begin{dcases}
    \,1\, &\text{if $i=j$ and $1\leq i \leq p_2$},
    \\
    -1\, &\text{if $i=j$ and $p_2+1\leq i \leq m_2$},
    \\
    \,0 &\text{if $i \neq j$}, 
  \end{dcases}
\end{align*}
namely, in such a way as the matrix representation of the scalar product $\left\langle \cdot, \cdot \right\rangle_{\mathfrak{v}_2}$ with respect to the  basis $\left(B_1, \ldots, B_{m_2}\right)$ is given by $\mathsf{E}_{p_2}\oplus \left(-\mathsf{E}_{q_2}\right)$.
To prove this, since $\mathfrak{v}_2$ can be decomposed into the direct sum of the eigenspaces $W_{\lambda_i} $ for $1\leq i \leq m_2$, it is enough to check that $W_{\lambda_1},\dots,W_{\lambda_{m_2}}$ are pairwise orthogonal to each other with respect to $\langle\cdot,\cdot\rangle_{\mathfrak{v}_2}$ and $W_{\lambda_i}$ is a non-degenerate subspace of $\left( \mathfrak{v}_2,\langle\cdot,\cdot\rangle_{\mathfrak{v}_2} \right)$ for any $i,\,1\leq i \leq m_2$. Since $\Phi^{\ast}\circ \Phi$ is a self-adjoint operator, for any $X\in W_{\lambda_i},Y\in W_{\lambda_j}$, we have $\left( \lambda_i-\lambda_j \right)\langle X,Y \rangle_{\mathfrak{v}_2}=0$. This implies that $\langle X,Y \rangle_{\mathfrak{v}_2}=0$ unless $i=j$, and thus the first condition holds. Next, to show that $W_{\lambda_i}$ is a non-degenerate subspace, we assume that $X\in W_{\lambda_i}$ satisfies $\langle X,Z^{\prime} \rangle_{\mathfrak{v}_2}=0$ for any $Z^{\prime}\in W_{\lambda_i} $. By the orthogonality of $W_{\lambda_i}$'s and the eigenspace decomposition $\mathfrak{v}_2=W_{\lambda_1}\oplus\dots\oplus W_{\lambda_{m_2}}$, we have $\langle X, Z\rangle_{\mathfrak{v}_2}=0$ for any $Z\in \mathfrak{v}_2$, which yields $X=0$ by the non-degeneracy of $\langle\cdot,\cdot\rangle_{\mathfrak{v}_2}$ on $\mathfrak{v}_2$.

Put $r\coloneq \mathrm{rank}\left( \Phi \right)$. 
By $d_1$ (resp. $d_2$) we denote the number of positive (resp. negative) eigenvalues $\lambda_i$ for $\Phi^{\ast}\circ\Phi$ with $1\leq i \leq p_2$, for which the corresponding eigenvectors $B_i$ satisfy $\left\langle B_i, B_i\right\rangle_{\mathfrak{v}_2}=1$. 
Similarly, the number $d_3$ (resp. $d_4$) stands for the number of negative (resp. positive) eigenvalues $\lambda_i$ for $\Phi^{\ast}\circ\Phi$ with $p_2+1\leq i \leq m_2$, which corresponds to $B_i$ satisfying $\left\langle B_i, B_i\right\rangle_{\mathfrak{v}_2}=-1$. 
Note that $1\leq d_1+ d_2\leq p_2$, $1\leq d_3+d_4\leq q_2$, and $r=d_1+d_2+d_3+d_4$. 
Without loss of generality, we can also assume that 
\begin{align*}
  \begin{dcases}
    \lambda_i > 0 &\text{if $1\leq i\leq d_1\,\text{or}\,p_2+d_3+1\leq i \leq p_2+d_3+d_4$},
    \\
    \lambda_i < 0 &\text{if $d_1+1\leq i \leq d_1+d_2\,\text{or}\,p_2+1\leq i \leq p_2+d_3$},
    \\
    \lambda_i =0 &\text{otherwise}. 
  \end{dcases}
\end{align*}
Note that $\Phi\left(B_1\right), \ldots, \Phi\left(B_{d_1+d_2}\right), \Phi\left(B_{p_2+1}\right), \ldots, \Phi\left(B_{p_2+d_3+d_4}\right)$ are linearly independent and span $\mathrm{Im}\left(\Phi\right)$. 
In fact, assuming 
\[
\sum_{i=1, \ldots, d_1+d_2, p_2+1, \ldots, p_2+d_3+d_4}c_i \Phi\left(B_i\right)=0, \quad c_i\in \mathbb{R}, 
\]
we have 
\[
\sum_{i=1, \ldots, d_1+d_2, p_2+1, \ldots, p_2+d_3+d_4}c_i \Phi^{\ast}\circ \Phi\left(B_i\right)=0, \quad 
\text{and hence}\, 
\sum_{i=1, \ldots, d_1+d_2, p_2+1, \ldots, p_2+d_3+d_4}c_i \lambda_i B_i=0.
\]
By the linear independence of $B_i$'s, we have $c_i\lambda_i=0$, $i=1, \ldots, d_1+d_2, p_2+1, \ldots, p_2+d_3+d_4$, which means $c_i=0$, $i=1, \ldots, d_1+d_2, p_2+1, \ldots, p_2+d_3+d_4$, as $\lambda_i\neq 0$, $i=1, \ldots, d_1+d_2, p_2+1, \ldots, p_2+d_3+d_4$. 

Taking the formula 
\begin{align*}
&\left\langle \dfrac{1}{\sqrt{\left|\lambda_i\right|}}\Phi\left(B_i\right), \dfrac{1}{\sqrt{\left|\lambda_i\right|}}\Phi\left(B_i\right)\right\rangle_{\mathfrak{v}_2}
=
\dfrac{1}{\left|\lambda_i\right|}\left\langle \Phi\left(B_i\right), \Phi\left(B_i\right)\right\rangle_{\mathfrak{v}_1}
=
\dfrac{1}{\left|\lambda_i\right|}\left\langle B_i, \Phi^{\ast}\circ \Phi\left(B_i\right)\right\rangle_{\mathfrak{v}_2} \\
&\qquad\qquad =
\dfrac{\lambda_i}{\left|\lambda_i\right|}\left\langle B_i, B_i\right\rangle_{\mathfrak{v}_2} 
=
\begin{dcases}
1 & \text{if $1\leq i\leq d_1\; \text{or}\; p_2+1\leq i \leq p_2+d_3$}, \\
-1 & \text{if $d_1+1\leq i \leq d_1+d_2\; \text{or}\; p_2+d_3+1\leq i \leq p_2+d_3+d_4$},
\end{dcases}
\end{align*}
into account, we set 
\begin{equation}\label{A_by_B}
A_i
\coloneq
\begin{dcases}
    \dfrac{1}{\sqrt{\left|\lambda_i\right|}}\Phi\left(B_i\right)\, &\text{if $1\leq i\leq d_1$},
    \\
    \dfrac{1}{\sqrt{\left|\lambda_{p_2+i-d_1}\right|}}\Phi\left(B_{p_2+i-d_1}\right)\, &\text{if $d_1+1\leq i\leq d_1+d_3$},
    \\
    \dfrac{1}{\sqrt{\left|\lambda_{d_1+i-p_1}\right|}}\Phi\left(B_{d_1+i-p_1}\right)\, &\text{if $p_{1}+1\leq i\leq p_{1}+d_2$},
    \\
    \dfrac{1}{\sqrt{\left|\lambda_{p_2+d_3+i-p_1-d_2}\right|}}\Phi\left(B_{p_2+d_3+i-p_1-d_2}\right)\, &\text{if $p_{1}+d_2+1 \leq  i\leq p_{1}+d_2+d_4$},
  \end{dcases}
\end{equation}
for $i=1, \ldots, d_1+d_2, p_2+1, \ldots, p_2+d_3+d_4$. 
Since $\mathrm{Im}\left(\Phi\right)\subset \mathfrak{v}_1$ is non-degenerate with respect to $\left\langle \cdot, \cdot\right\rangle_{\mathfrak{v}_1}$, we can extend the vectors $A_i$, $i=1, \ldots, d_1+d_2, p_2+1, \ldots, p_2+d_3+d_4$, to a basis $A_i$, $1\leq i\leq m_1$, of $\mathfrak{v}_1$ such that 
\begin{align*}
  \langle A_i, A_j\rangle_{\mathfrak{v}_1}=
  \begin{dcases}
    \,1\, &\text{if $i=j$ and $1\leq i\leq p_1$},
    \\
    -1\, &\text{if $i=j$ and $p_1+1\leq i \leq m_1$},
    \\
    \,0\, &\text{if $i\neq j$}.
  \end{dcases}
\end{align*}
This is equivalent to say that the matrix representation of the scalar product $\left\langle \cdot, \cdot \right\rangle_{\mathfrak{v}_1}$ with respect to the  basis $\left(A_1, \ldots, A_{m_1}\right)$ is given by $\mathsf{E}_{p_1}\oplus \left(-\mathsf{E}_{q_1}\right)$. 

By \eqref{A_by_B}, we have 
\[
\Phi\left(B_i\right)
=
\begin{dcases}
    \sqrt{\left|\lambda_i\right|}A_i\, &\text{if $1\leq i\leq d_1$},
    \\
    \sqrt{\left|\lambda_{i}\right|}A_{p_1+i-d_1}\, &\text{if $d_1+1\leq i\leq d_1+d_2$},
    \\
    \sqrt{\left|\lambda_{i}\right|}A_{d_1+i-p_2}\, &\text{if $p_2+1\leq i\leq p_{2}+d_3$},
    \\
    \sqrt{\left|\lambda_{i}\right|}A_{p_1+d_2+i-p_2-d_3}\, &\text{if $p_{2}+d_3+1 \leq i\leq p_{2}+d_3+d_4$}.
  \end{dcases}
\]
Thus, we obtain the matrix representation of $\Phi=j_B^{12}(Z_0)$ with respect to the basis $\left(A_1, \ldots, A_{m_1}\right)$ of $\mathfrak{v}_1$ and the one $\left(B_1, \ldots, B_{m_2}\right)$ of $\mathfrak{v}_2$ as 
\begin{align*}
  \begin{array}{cccc:ccccc}
    \ldelim({6}{4mm} & D_1 & 0 & 0 & 0 & 0 & 0 & \rdelim){6}{1mm} & \rdelim\}{3}{*}[$p_1$] \\
    & 0 & 0 & 0 & D_3 & 0 & 0 & \\
    & 0 & 0 & 0 & 0 & 0 & 0 & \\ \cdashline{2-7}
    & 0 & D_2 & 0 & 0 & 0 & 0 & & \rdelim\}{3}{*}[$q_1$] \\
    & 0 & 0 & 0 & 0 & D_4 & 0 & \\ 
    & 0 & 0 & 0 & 0 & 0 & 0 & \\
    & \multicolumn{3}{l}{\underbrace{\hspace{12ex}}_{\begin{matrix} p_2 \end{matrix}}} & \multicolumn{3}{l}{\underbrace{\hspace{12ex}}_{\begin{matrix} q_2 \end{matrix}}}
  \end{array}. 
\end{align*}
Here, the diagonal matrices $D_1,D_2,D_3,D_4$ are written as 
\begin{align*}
D_1
&=
\mathrm{diag}\left(\sqrt{\left|\lambda_1\right|}, \ldots, \sqrt{\left|\lambda_{d_1}\right|}\right), \\
D_2
&=
\mathrm{diag}\left(\sqrt{\left|\lambda_{d_1+1}\right|}, \ldots, \sqrt{\left|\lambda_{d_1+d_2}\right|}\right), \\
D_3
&=
\mathrm{diag}\left(\sqrt{\left|\lambda_{p_2+1}\right|}, \ldots, \sqrt{\left|\lambda_{p_2+d_3}\right|}\right), \\
D_4
&=
\mathrm{diag}\left(\sqrt{\left|\lambda_{p_2+d_3+1}\right|}, \ldots, \sqrt{\left|\lambda_{p_2+d_3+d_4}\right|}\right).
\end{align*}

Consequently, we have the matrix representation of the operator $\begin{pmatrix} 0 & -j_B^{12}\left(Z_0\right) \\ j_B^{21}\left(Z_0\right) & 0 \end{pmatrix}$ with respect to the above basis as 
\begin{align*}
  \begin{pmatrix}
    0 & 0 & 0 & 0 & 0 & 0 & -D_1 & 0 & 0 & 0 & 0 & 0 \\
    0 & 0 & 0 & 0 & 0 & 0 & 0 & 0 & 0 & -D_3 & 0 & 0\\
    0 & 0 & 0 & 0 & 0 & 0 & 0 & 0 & 0 & 0 & 0 & 0\\
    0 & 0 & 0 & 0 & 0 & 0 & 0 & -D_2 & 0 & 0 & 0 & 0 \\
    0 & 0 & 0 & 0 & 0 & 0 & 0 & 0 & 0 & 0 & -D_4 & 0 \\
    0 & 0 & 0 & 0 & 0 & 0 & 0 & 0 & 0 & 0 & 0 & 0 \\
    D_1 & 0 & 0 & 0 & 0 & 0 & 0 & 0 & 0 & 0 & 0 & 0 \\
    0 & 0 & 0 & -D_2 & 0 & 0 & 0 & 0 & 0 & 0 & 0 & 0 \\
    0 & 0 & 0 & 0 & 0 & 0 & 0 & 0 & 0 & 0 & 0 & 0 \\
    0 & -D_3 & 0 & 0 & 0 & 0 & 0 & 0 & 0 & 0 & 0 & 0 \\
    0 & 0 & 0 & 0 & D_4 & 0 & 0 & 0 & 0 & 0 & 0 & 0 \\
    0 & 0 & 0 & 0 & 0 & 0 & 0 & 0 & 0 & 0 & 0 & 0 \\
  \end{pmatrix},
\end{align*}
for which the number of purely imaginary, real, complex, and zero eigenvalues, respectively, is given as the quadruple $\displaystyle \left( 2(d_1+d_4), 2(d_2+d_3), 0, m_1+m_2-2r\right)$.

We recall the Lie algebra $\mathfrak{g}=\mathfrak{u}_1\dot{+}\mathfrak{u}_2\dot{+}\mathfrak{w}$ of the Heisenberg-Reiter group, equipped with a  biliear form $B: \mathfrak{u}_1\times \mathfrak{u}_2\rightarrow \mathfrak{w}$ induced by a Lie bracket on $\mathfrak{g}$, was discussed in \S \ref{subsec_HR}. The centre $\mathfrak{z}$ of $\mathfrak{g}$ is given as $\mathfrak{z}=\mathfrak{u}_1^{\circ}\dot{+}\mathfrak{u}_2^{\circ}\dot{+}\mathfrak{w}$, where $\mathfrak{u}_1^{\circ}=\left\{U_1\in \mathfrak{u}_1\mid \forall U_2\in \mathfrak{u}_2,\; B\left(U_1, U_2\right)=0\right\}$ and $\mathfrak{u}_2^{\circ}=\left\{U_2\in \mathfrak{u}_2\mid \forall U_1\in \mathfrak{u}_1,\; B\left(U_1, U_2\right)=0\right\}$. With this in mind, we write $Y_{\mathfrak{z}}=\overline{Y_{\mathfrak{z}}}+\left(Y_{\mathfrak{z}}\right)_{0}\in\mathfrak{z}=\mathfrak{u}_1^{\circ}\dot{+}\mathfrak{u}_2^{\circ}\dot{+}\mathfrak{w}$ with $\overline{Y_{\mathfrak{z}}}\in \mathfrak{u}_1^{\circ}\dot{+}\mathfrak{u}_2^{\circ}$ and $\left(Y_{\mathfrak{z}}\right)_{0}\in\mathfrak{w}$ for the equilibrium points $Y=Y_{\mathfrak{v}}+Y_{\mathfrak{z}}$, where $Y_{\mathfrak{v}}\in\ker\left(j\left(Y_{\mathfrak{z}}\right)\right)$, for the Lie-Poisson equation. Then, the stability of equilibria for the Heisenberg-Reiter group is charactarized as follows:
\begin{theorem}\label{hei_rei}
For the Heisenberg-Reiter group, if the operator $j_B^{21}\left(\left(	Y_{\mathfrak{z}}\right)_{0}\right) j_B^{12}\left(\left(	Y_{\mathfrak{z}}\right)_{0}\right)\colon\mathfrak{v}_2\to\mathfrak{v}_{2}$ is diagonalizable over the real number field and if the image of $j_B^{12}\left(\left(	Y_{\mathfrak{z}}\right)_{0}\right)\colon\mathfrak{v}_2\to\mathfrak{v}_{1}$ is a non-degenerate subspace of $\left( \mathfrak{v}_1,\langle\cdot,\cdot\rangle_{\mathfrak{v}_1} \right)$, the equilibrium point $Y_{\mathfrak{v}}+Y_{\mathfrak{z}}$, where $Y_{\mathfrak{v}}\in\ker\left(j\left(Y_{\mathfrak{z}}\right)\right)$, for the Lie-Poisson equation described in the first component of \eqref{eq_geodesic_flow} has the Williamson type $\left(d_1+d_4, d_2+d_3, 0 \right)$. 
  \hfill $\square$
\end{theorem}

\paragraph{Step-two nilpotent Lie groups associated to semi-simple modules}
Let $G$ be a step-two nilpotent Lie group whose Lie algebra $\left( \mathfrak{g},\left[ \cdot,\cdot \right] \right)$ is a real step-two nilpotent Lie algebra arising from a given semi-simple module $\left( \mathfrak{v},j \right)$. 
Then, the Lie algebra $\mathfrak{g}$ has the direct sum decomposition $\mathfrak{g}=\mathfrak{v}\dot{+}\mathfrak{z}$ equipped with an inner product $\langle\cdot,\cdot\rangle$ satisfying the three conditions described in \S \ref{Step-two Lie groups associated to semi-simple modules}.

From the condition (E\ref{cond_e1}) in \S \ref{Step-two Lie groups associated to semi-simple modules}, the restriction $\left\langle \cdot, \cdot \right\rangle_{\mathfrak{v}}:=\left\langle \cdot, \cdot \right\rangle|_{\mathfrak{v}\times \mathfrak{v}}$ is again an inner product on $\mathfrak{v}$.
Moreover, from the condition (E\ref{cond_e2}), we can regard the operators $j\left( Z \right)$ for $Z\in\mathfrak{z}$ as skew-symmetric matrices with respect to a suitable basis in $\mathfrak{v}$. Then, all the eigenvalues of $j\left( Z \right)$ for $Z\in\mathfrak{z}$ are purely imaginary, or equivalently $j\left( Z \right)$ for $Z\in\mathfrak{z}$ has only elliptic component. 
Thus, we have the followings: 
\begin{proposition}\label{ssm}
The equilibrium points of the Lie-Poisson equation \eqref{eq_geodesic_flow} for step-two nilpotent Lie groups associated to semi-simple modules are elliptic and hence Lyapunov stable. 
\hfill $\square$
\end{proposition}

\section{Concluding Remarks}
By the general framework of Lie-Poisson Reduction, the Lie-Poisson equation \eqref{eq_geodesic_flow} is obtained for the left-invariant pseudo-Riemannian geodesic flow of a step-two Lie group in \S \ref{subsec_lp}. 
As is known, the Lie-Poisson equation \eqref{eq_geodesic_flow} can be restricted to a coadjoint orbit and the restricted Hamilton equation on a generic orbit $\mathcal{O}$ is given as a linear differential equation \eqref{lie-poisson on orbits}. 
The reduced equation has only one equilibrium point $Y_{\mathfrak{v}}+Y_{\mathfrak{z}}\in \mathcal{O}$, where $Y_{\mathfrak{v}}\in \ker\left(j\left(Y_{\mathfrak{z}}\right)\right)$. 
In \S \ref{subsec_wt_gen}, the stability of the equilibrium is charactarized based on the eigenvalues of the operator $j\left( Y_{\mathfrak{z}} \right)\in\mathfrak{so}\left( \mathfrak{v},\langle\cdot,\cdot\rangle_{\mathfrak{v}} \right)$ for the general step-two nilpotent Lie groups. 
The precise Williamson types of equilibrium points for several concrete classes of step-two nilpotent Lie groups in \S \ref{Certain classes of step-two nilpotent Lie groups} are described in \S \ref{subsec_wt_conc}. 
The main results of the present paper are summarized as follows:
\begin{itemize}
\item For general step-two nilpotent Lie groups, the Williamson type of the equilibrium is determined in Theorem \ref{main_thm_general} on the basis of the classification results on the conjugacy classes of Cartan subalgebras of semi-simple Lie algebras of types $\mathsf{B}$ and $\mathsf{D}$ discussed in \cite{sugiura_1959}.
\item In the case of a step-two Carnot group, the Williamson type of the equilibrium is found in Corollary \ref{cor_carnot}. 
In particular, it is shown that the equilibrium for M\'etivier groups is Lyapunov stable.
\item Theorem \ref{stability_pseudo-h-type} (respectively Corollary \ref{cor_ht}) describes the stability of the equilibrium for a pseudo-$H$-type (respectively $H$-type) Lie group.
\item For the Heisenberg-Reiter Lie groups, the Williamson type of the equilibrium is given in Theorem \ref{hei_rei} on the basis of the singular value decomposition with respect to an indefinite scalar product discussed in \cite{hassi_1991}.
\item In the case of step-two nilpotent Lie groups associated to semi-simple modules, the equilibrium is shown to be Lyapunov stable in Proposition \ref{ssm}. 
\end{itemize}

The present paper is concentrated only on step-two nilpotent Lie groups and the main idea is to discuss the ``$j$-mapping.'' 
It is interesting to investigate the stability of equilibrium points for more general nilpotent Lie groups of higher steps.
Such a problem may be studied on the basis of other techniques. 

\bigskip

\noindent \textbf{Acknowledgements:} 
The authers thank the referees for the helpful comments to improve the quality of the paper.
The second author would like to thank Tudor S. Ratiu and Hiroshi Tamaru for their valuable discussions particularly in relation to the classification of Cartan subalgebras. 
The first author is partially supported by JST SPRING Grant Number JPMJSP2101. 
The second author is partially supported by JSPS KAKENHI Grant Numbers 23K22409, 23H04481, 24K06749, 25H01492.
This work was partially supported by the Research Institute for Mathematical Sciences, an International Joint Usage/Research Center located in Kyoto University.


\bibliographystyle{plain}
\bibliography{references_corrected}


\appendix
\section{Appendix: Classification for the conjugacy classes of Cartan subalgebras in non-compact real forms of complex simple Lie algebras of types $\mathsf{B}_n$ and $\mathsf{D}_n$.}\label{appendix_a} 
In \cite{kostant_1955,kostant_2009}, a classification of conjugacy classes of Cartan subalgebras in non-compact simple Lie algebras is given by Kostant. 
Later by Sugiura in \cite{sugiura_1959}, the classification problem was dealt with independently. 
More precisely, the former classification is about the conjugacy classes up to full automorphism group of the Lie algebras, whereas the latter is about those up to inner automorphisms. 
Here, we describe the classification results in the case of simple Lie algebras of types $\mathsf{B}_n$ and $\mathsf{D}_n$ along the line of \cite{sugiura_1959}.

We first summarize the general method of the classification which amounts to the notion of admissible root systems. 
Let $\mathfrak{g}$ be a real semi-simple Lie algebra equipped with the Killing form $B$ and $\mathfrak{h}_0\subset \mathfrak{g}$ a fixed Cartan subalgebra. 
For a subalgebra $\mathfrak{k}\subset \mathfrak{g}$, a vector subspace $\mathfrak{p}\subset \mathfrak{g}$, and an Abelian  subalgebra $\mathfrak{m}\subset \mathfrak{g}$, the triple $\left(\mathfrak{k}, \mathfrak{p}, \mathfrak{m}\right)$ is called a standard triple if $\mathfrak{g}=\mathfrak{k}\dot{+}\mathfrak{p}$ is the Cartan decomposition of $\mathfrak{g}$ and $\mathfrak{m}$ is a maximal Abelian subalgebra contained in $\mathfrak{p}$. 
A Cartan subalgebra $\mathfrak{h}\subset \mathfrak{g}$ is called standard with respect to the standard triple $\left(\mathfrak{k}, \mathfrak{p}, \mathfrak{m}\right)$ if its vector part $\mathfrak{h}^-$ is in $\mathfrak{m}$ and the toroidal part $\mathfrak{h}^+$ is in $\mathfrak{k}$, where we write 
\begin{align*}
\mathfrak{h}^+:&=\left\{X\in\mathfrak{h}\mid \text{eigenvalues of}\, \mathrm{ad}_X\, \text{are purely imaginary}\right\}, \\
\mathfrak{h}^-:&=\left\{X\in\mathfrak{h}\mid \text{eigenvalues of}\, \mathrm{ad}_X\, \text{are real}\right\}. 
\end{align*}
Now, for a vector subspace $\mathfrak{n}\subset \mathfrak{m}$, we set $\mathfrak{l}:=\mathfrak{n}^{\perp}\cap \mathfrak{m}=\left\{X\in \mathfrak{m}\mid B\left(X, \mathfrak{n}\right)=0\right\}$. 
Then, a set $\bm{F}=\left\{\alpha_1, \ldots, \alpha_{\ell}\right\}$ of roots is called an \textit{admissible root system} if $\alpha_i\pm\alpha_j$ are not roots for any $i, j=1, \ldots, \ell$, $\alpha_i\pm\alpha_j\neq 0$ if $i\neq j$, and $\mathfrak{l}=\displaystyle \sum_{i=1}^{\ell}\mathbb{R}H_{\alpha_i}$, where for any root $\lambda$, $H_{\lambda}\in\mathfrak{h}$ is defined through $\lambda\left(H\right)=B\left(H_{\lambda}, H\right)$, $\forall H\in\mathfrak{h}$. 
We quote the following theorem by Sugiura \cite[Theorem 5, pp.394-395]{sugiura_1959}. 
\begin{theorem}[\cite{sugiura_1959}]
There exists a standard Cartan subalgebra such that $\mathfrak{h}^-=\mathfrak{n}$, if and only if there exists an admissible root system $\bm{F}=\left\{\alpha_1, \ldots, \alpha_{\ell}\right\}$ such that $\mathfrak{l}=\displaystyle \sum_{i=1}^{\ell}\mathbb{R}H_{\alpha_i}$. 
\end{theorem}
Due to this theorem, the classification of the conjugacy classes of the Cartan subalgera of a given semi-simple Lie algebra amounts to the classification for conjugacy classes of the admissible root systems. 
(See \cite[p. 395, Theorem 6]{sugiura_1959}.) 
Here, two admissible root systems $\bf{F}_1$ and $\bf{F}_2$ are called equivalent if 
\[
\mathrm{span}_{\mathbb{R}}\left\{H_{\alpha}\mid \alpha \in \bf{F}_1\right\}
=
\mathrm{span}_{\mathbb{R}}\left\{H_{\beta}\mid \beta \in \bf{F}_2\right\}. 
\]
In this case, we write $\bf{F}_1\equiv \bf{F}_2$. 
Further, two admissible root systems are called conjugate if there exists an element $\sigma$ of the Weyl group for $\mathfrak{g}\otimes \mathbb{C}$ with respect to the Cartan subalgebra $\mathfrak{h}_0\otimes \mathbb{C}$ such that $\sigma\bf{F}_1=\bf{F}_2$. 

In \cite{sugiura_1959}, this classification is done for each of non-compact real simple Lie algebras. 
Here, we consider the case where
\begin{align*}
\mathfrak{g}
&=
\mathfrak{o}\left(m,m+p\right)
=
\left\{X\in\mathbb{R}^{s\times s}\mid X^{\mathrm{T}}B_m+B_mX=0\right\} \\
&=
\left\{\left.\begin{pmatrix}
A & B & D \\
C & -A^{\mathrm{T}} & -D \\
F^{\mathrm{T}} & D^{\mathrm{T}} & L
\end{pmatrix}\right|A\in\mathbb{R}^{m\times m}, B, C\in\mathfrak{o}(m), D,F\in\mathbb{R}^{m\times p}, L\in\mathfrak{o}(p)\right\},
\end{align*}
where $B_m=\begin{pmatrix}
0 & \mathsf{E}_m & 0 \\
\mathsf{E}_m & 0 & 0 \\
0 & 0 & -\mathsf{E}_p 
\end{pmatrix}$, $s=2m+p$, $\mathsf{E}_m$, $\mathsf{E}_p$ are respectively the $m\times m$ and the $p\times p$ unit matrices, and $\mathfrak{o}(m)$, $\mathfrak{o}(p)$ are respectively the set of all $m\times m$ skew-symmetric matrices and the one of all $p\times p$ skew-symmetric matrices. 
The rank $n$ of the Lie algebra $\mathfrak{g}\otimes \mathbb{C}$ is given by the formula: 
\[
s
=
\begin{cases}
2n+1 &\text{if}\,\mathfrak{g}\otimes \mathbb{C}\,\text{is of type}\, \mathsf{B}_n, \\
2n &\text{if}\,\mathfrak{g}\otimes \mathbb{C}\,\text{is of type}\, \mathsf{D}_n. 
\end{cases}
\]

In this case, the Cartan decomposition of $\mathfrak{g}=\mathfrak{o}\left(m,m+p\right)=\mathfrak{k}\dot{+}\mathfrak{p}$ is given by 
\begin{align*}
\mathfrak{k}
&=
\left\{\left.\begin{pmatrix}
A & B & D \\
B & A & -D \\
-D^{\mathrm{T}} & D^{\mathrm{T}} & L
\end{pmatrix}\right|A, B\in\mathfrak{o}(m), D\in\mathbb{R}^{m\times p}, L\in\mathfrak{o}(p)\right\}, \\
\mathfrak{p}
&=
\left\{\left.\begin{pmatrix}
A & B & D \\
-B & -A & D \\
D^{\mathrm{T}} & D^{\mathrm{T}} & 0
\end{pmatrix}\right|A\in\mathrm{Sym}\left(m, \mathbb{R}\right), B\in\mathfrak{o}(m), D\in\mathbb{R}^{m\times p}\right\}, 
\end{align*}
where $\mathrm{Sym}\left(m, \mathbb{R}\right)$ stands for the set of all $m\times m$ symmetric matrices. 
A maximal Abelian subalgebra $\mathfrak{m}$ in $\mathfrak{p}$ is given as 
\[
\mathfrak{m}
=
\left\{\mathrm{diag}\left(h_1, \ldots, h_m, -h_1, \ldots, -h_m, 0, \ldots, 0\right)\mid h_1, \ldots, h_m\in\mathbb{R}\right\}. 
\]
We take the standard Cartan subalgebra 
\begin{align*}
&\mathfrak{h}_0=
\Biggr\{\left.
H=
\mathrm{diag}\left(h_1, \ldots, h_m, -h_1, \ldots, -h_m\right)
\oplus
\begin{pmatrix}
0 & \mathrm{diag}(-u_1, \ldots, -u_t) \\
\mathrm{diag}(u_1, \ldots, u_t) & 0
\end{pmatrix}\right| \\
&\qquad\qquad\qquad\qquad\qquad\qquad\qquad\qquad\qquad\qquad\qquad\qquad\qquad\qquad h_1, \ldots, h_m, u_1, \ldots, u_t\in\mathbb{R}\Biggr\}
\end{align*}
if $\mathfrak{g}$ is of type $\mathsf{D}_n$
and 
\begin{align*}
&\mathfrak{h}_0=
\Biggr\{\left.
H=
\mathrm{diag}\left(h_1, \ldots, h_m, -h_1, \ldots, -h_m\right)
\oplus
\begin{pmatrix}
0 & \mathrm{diag}(-u_1, \ldots, -u_t) \\
\mathrm{diag}(u_1, \ldots, u_t) & 0
\end{pmatrix}
\oplus 0\right| \\
&\qquad\qquad\qquad\qquad\qquad\qquad\qquad\qquad\qquad\qquad\qquad\qquad\qquad\qquad\qquad h_1, \ldots, h_m, u_1, \ldots, u_t\in\mathbb{R}\Biggr\}
\end{align*}
if $\mathfrak{g}$ is of type $\mathsf{B}_n$, where $t=\left\lfloor\dfrac{s-2m}{2}\right\rfloor$. 
(The definition of $t$ in \cite[p. 403]{sugiura_1959} should be read as this formula. )

We next take linear functions $e_i\in\left(\mathfrak{h}_0\otimes \mathbb{C}\right)^{\ast}$ defined through 
\[
e_i\left(H\right)
=
\begin{cases}
h_i, & i=1, \ldots, m, \\
\sqrt{-1}u_{i-m}, & i=m+1, \ldots, n,
\end{cases}
\]
where $H\in\mathfrak{h}_0$. 
The root system for $\mathfrak{g}\otimes \mathbb{C}$ is then given by 
\[
\Delta
=
\begin{cases}
\left\{\pm\left(e_i\pm e_j\right)\mid 1\leq i<j\leq n\right\}\cup\left\{\pm e_i\mid i=1, \ldots, n\right\}, &\text{if}\, \mathfrak{g}\, \text{is of type}\, \mathsf{B}_n, \\
\left\{\pm\left(e_i\pm e_j\right)\mid 1\leq i<j\leq n\right\}, &\text{if}\, \mathfrak{g}\, \text{is of type}\, \mathsf{D}_n. 
\end{cases}
\]
We below describe the conjugacy classes of Cartan subalgebras separately in the case of types $\mathsf{B}_n$ and $\mathsf{D}_n$. 

\paragraph{Conjugacy classes of Cartan subalgebras for $\mathfrak{g}$ of type $\mathsf{D}_n$.} 

The admissible roots are classified as one of the sets 
\[
\bm{F}\left(\ell, k\right)
=
\left\{e_1+e_2, e_1-e_2, \ldots, e_{2\ell-1}+e_{2\ell}, e_{2\ell-1}-e_{2\ell}, e_{2\ell+1}-e_{2\ell+k+1}, \ldots, e_{2\ell+k}-e_{2\ell+2k}\right\}, 
\]
where $0\leq k$, $0\leq \ell$, $2\left(k+\ell\right)\leq m$.
If $n=m$ and $m$ is even, we have to add one another set $\bm{F}\left(0,(m-2)/2\right)\cup\left\{e_{m-1}+e_m\right\}$. 
In this case, we have $p=0$. 
Note that \cite[(76), p.403]{sugiura_1959} should be read as $\bm{F}\left(0,(m-2)/2\right)\cup\left\{e_{m-1}+e_m\right\}$ to have a correct size of set of roots. 

For the admissible root system $\bm{F}\left(\ell, k\right)$, the corresponding conjugacy class of Cartan subalgebras is represented by the set $\mathfrak{h}$ of all the following matrices: 
\begin{equation}\label{Cartan_D_n}
\begin{pmatrix}
D_1\oplus \left(D_3+D_4\right)\oplus D_5 & D_2\oplus 0\oplus 0 \\
D_2\oplus 0\oplus 0 & D_1\oplus \left(-D_3+D_4\right)\oplus \left(-D_5\right)
\end{pmatrix}
\oplus D_6,
\end{equation}
where 
\begin{align}\label{matrices_D_n}
D_1
&=
\begin{pmatrix}
0 & \mathrm{diag}\left(-h_1, \ldots, -h_{\ell}\right) \\
\mathrm{diag}\left(h_1, \ldots, h_{\ell}\right) & 0
\end{pmatrix}, \notag \\
D_2
&=
\begin{pmatrix}
0 & \mathrm{diag}\left(-h_{\ell +1}, \ldots, -h_{2\ell}\right) \\
\mathrm{diag}\left(h_{\ell +1}, \ldots, h_{2\ell}\right) & 0
\end{pmatrix}, \notag \\
D_3
&=
\mathrm{diag}\left(h_{2\ell +1}, \ldots, h_{2\ell +k}, h_{2\ell +1}, \ldots, h_{2\ell +k}\right), \notag \\
D_4
&=
\begin{pmatrix}
0 & \mathrm{diag}\left(-h_{2\ell +k+1}, \ldots, -h_{2\ell+2k}\right) \\
\mathrm{diag}\left(h_{2\ell +k+1}, \ldots, h_{2\ell+2k}\right) & 0
\end{pmatrix}, \notag \\
D_5
&=
\mathrm{diag}\left(h_{2\ell+2k+1}, \ldots, h_m\right), \notag \\
D_6
&=
\begin{pmatrix}
0 & \mathrm{diag}\left(-h_{m +1}, \ldots, -h_{n}\right) \\
\mathrm{diag}\left(h_{m +1}, \ldots, h_{n}\right) & 0
\end{pmatrix}. 
\end{align}
For the admissible root system $\bm{F}\left(0,\dfrac{m-2}{2}\right)\cup\left\{e_{m-1}+e_m\right\}$, which appears if $m=n$ and $m$ is even, the corresponding conjugacy class of Cartan subalgebras is represented by the set $\mathfrak{h}$ of all the following matrices: 
\begin{equation}\label{Cartan_D_n_bis}
\begin{pmatrix}
D_3+D_4 & 0 & 0 & 0 \\
0 & D_5^{\prime} & 0 & D_7 \\
0 & 0 & -D_3+D_4 & 0 \\
0 & D_7 & 0 & D_5^{\prime}
\end{pmatrix},
\end{equation}
where 
\begin{align}\label{matrices_D_n_bis}
D_3
&=
\mathrm{diag}\left(h_{1}, \ldots, h_{(m-2)/2}, h_{1}, \ldots, h_{(m-2)/2}\right), \notag \\
D_4
&=
\begin{pmatrix}
0 & \mathrm{diag}\left(-h_{m/2}, \ldots, -h_{m-2}\right) \\
\mathrm{diag}\left(h_{m/2}, \ldots, h_{m-2}\right) & 0
\end{pmatrix}, \notag \\
D_5^{\prime}
&=
\mathrm{diag}\left(h_{m-1}, -h_{m-1}\right), \notag \\
D_7
&=
\begin{pmatrix}
0 & -h_m \\
h_m & 0
\end{pmatrix}, 
\end{align}
and we have $\ell=0$ and $k=\dfrac{m-2}{2}$.

Note that the matrices in \eqref{Cartan_D_n} and in \eqref{matrices_D_n} are slightly different from the ones in \cite[(77), (78)]{sugiura_1959}. 
In order to ensure that $\mathfrak{l}=\mathfrak{n}^{\perp}\cap\mathfrak{m}=\left(\mathfrak{h}^-\right)^{\perp}\cap\mathfrak{m}$ consists of $0\oplus D_3\oplus D_5\oplus 0 \oplus \left(-D_3\right)\oplus \left(-D_5\right)\oplus 0$ and that $\mathfrak{h}$ be maximally Abelian in $\mathfrak{g}$, however, \cite[(77), (78)]{sugiura_1959} should be read as in \eqref{Cartan_D_n} and \eqref{matrices_D_n}, particularly about the components of $D_1$ in \eqref{matrices_D_n} and the position of $D_4$ in \eqref{Cartan_D_n}. 
Similarly, the representative for the conjugacy class given by the admissible root system $\bm{F}\left(0,\dfrac{m-2}{2}\right)\cup\left\{e_{m-1}+e_m\right\}$, which is explained in \cite[Last paragraph of p.404]{sugiura_1959}, should be read as in \eqref{Cartan_D_n_bis} and \eqref{matrices_D_n_bis}. 

\paragraph{Conjugacy classes of Cartan subalgebras for $\mathfrak{g}$ of type $\mathsf{B}_n$.}
The admissible root systems in this case are listed \cite[p.405]{sugiura_1959} as 
\begin{align*}
&\bm{F}\left(\ell, k\right) \\
&=
\left\{e_1+e_2, e_1-e_2, \ldots, e_{2\ell-1}+e_{2\ell}, e_{2\ell-1}-e_{2\ell},e_{2\ell +1}-e_{2\ell +k+1}, \ldots, e_{2\ell+k}-e_{2\ell+2k}\right\}, \\
&\qquad\qquad \qquad 0\leq k, 0\leq \ell, 2\left(k+\ell\right)\leq m; \\
&\bm{F}^{\prime}\left(\ell, k\right) \\
&=
\left\{e_1+e_2, e_1-e_2, \ldots, e_{2\ell-1}+e_{2\ell}, e_{2\ell-1}-e_{2\ell},e_{2\ell +1}-e_{2\ell +k+1}, \ldots, e_{2\ell+k}-e_{2\ell+2k}, e_{2\ell+2k+1}\right\}, \\
&\qquad\qquad \qquad 0\leq k, 0\leq \ell, 2\left(k+\ell\right)+1\leq m. 
\end{align*}
The corresponding conjugacy classes for $\bm{F}\left(\ell, k\right)$ are as in \eqref{matrices_D_n} and those for $\bm{F}^{\prime}\left(\ell, k\right)$ are represented by the sets of following matrices: 
\begin{equation}\label{Cartan_B_n}
\begin{pmatrix}
D_1 & 0 & 0 & D_2 & 0 & 0 & 0 & 0 \\
0 & D_3+D_4 & 0 & 0 & 0 & 0 & 0 & 0 \\
0 & 0 & D_5^{\prime\prime} & 0 & 0 & 0 & 0 & - D_8 \\
D_2 & 0 & 0 & D_1 & 0 & 0 & 0 & 0 \\
0 & 0 & 0 & 0 & -D_3+D_4  & 0 & 0 & 0 \\
0 & 0 & 0 & 0 & 0 & -D_5^{\prime\prime} & 0 & D_8 \\
0 & 0 & 0 & 0 & 0 & 0 & D_6 & 0 \\
0 & 0 & D_8^{\mathrm{T}} & 0 & 0 & -D_8^{\mathrm{T}} & 0 & 0
\end{pmatrix},
\end{equation}
where $D_1, D_2, D_3, D_4, D_6$ are the same matrices as in \eqref{matrices_D_n} with zero vectors being added to the last row and the last column, $D_5^{\prime\prime}=\mathrm{diag}\left(0, h_{2\left(\ell +k+1\right)}, \ldots, h_m\right)$, and $D_8=\left(h_{2\ell +2k+1}, 0, \ldots, 0\right)^{\mathrm{T}}\in \mathbb{R}^{m-2\ell-2k}$. 
Note that the matrix (77) in \cite[p.404]{sugiura_1959} should be read as in \eqref{Cartan_B_n}, particularly about the positions of $D_4$ and $D_8$. 
\end{document}